\documentclass[11pt, reqno]{amsart}

\usepackage{pstricks}
\usepackage{times}

\usepackage[numbers, sort&compress]{natbib}

\usepackage[breaklinks]{hyperref}

\usepackage{amsmath,amssymb,amsthm,mathtools}
\usepackage{hyperref}
\usepackage{enumitem}

\hypersetup{colorlinks=true, linkcolor=blue, citecolor=blue, urlcolor=blue}

\newtheorem{theorem}{Theorem}
\newtheorem{proposition}{Proposition}
\newtheorem{lemma}{Lemma}
\newtheorem{corollary}{Corollary}
\theoremstyle{remark}

\newcommand{\R}{\mathbb{R}}
\newcommand{\N}{\mathbb{N}}
\newcommand{\Z}{\mathbb{Z}}
\newcommand{\E}{\mathbb{E}}

\newcommand{\abs}[1]{\left\lvert #1 \right\rvert}
\newcommand{\norm}[1]{\left\lVert #1 \right\rVert}
\newcommand{\what}[1]{\widehat{#1}}

\newcommand{\Var}{\mathrm{Var}}

\usepackage[utf8]{inputenc}
\usepackage{amsmath}
\usepackage{amsfonts}
\usepackage{bbm}
\usepackage{mathtools}
\usepackage{tikz}
\usetikzlibrary{decorations.pathreplacing}
\usepackage{amsthm}
\usepackage[english]{babel}
\usepackage{fancyhdr}
\usepackage{amssymb}
\usepackage{enumitem}
\usepackage{qtree}
\usepackage{mathrsfs}

\begin{document}
\title[Orthogonality of Mat\'ern fields on bounded domains in $\mathbb{R}^4$]{Microergodicity implies orthogonality of Mat\'ern fields on bounded domains in $\mathbb{R}^4$}
\author{Natesh S. Pillai}

\address{Department of Statistics, Harvard University}
\email{pillai@fas.harvard.edu}

\begin{abstract}
Mat\'ern random fields are one of the most widely used classes of models in spatial statistics.
The fixed-domain identifiability of covariance parameters for stationary
Mat\'ern Gaussian random fields exhibits a dimension-dependent phase transition.
For known smoothness $\nu$, Zhang \cite{Zhang2004} showed that when $d\le3$,
two Mat\'ern models with the same microergodic parameter
$m=\sigma^2\alpha^{2\nu}$ induce equivalent Gaussian measures on bounded
domains, while Anderes \cite{Anderes2010} proved that when $d>4$, the
corresponding measures are mutually singular whenever the parameters differ.
The critical case $d=4$ for stationary Mat\'ern models has remained open.

We resolve this case. Let $d=4$ and consider two stationary Mat\'ern models on
$\mathbb R^4$ with parameters $(\sigma_1,\alpha_1)$ and $(\sigma_2,\alpha_2)$
satisfying
\[
\sigma_1^2\alpha_1^{2\nu}=\sigma_2^2\alpha_2^{2\nu},
\qquad
\alpha_1\neq \alpha_2.
\]
We prove that the corresponding Gaussian measures on any bounded observation
domain are mutually singular on every countable dense observation set, and
on the associated path space of continuous functions. 

Our approach can be viewed as a spectral analogue of the higher-order increment
method of Anderes \cite{Anderes2010}. Whereas Anderes isolates the second
irregular covariance coefficient through renormalized quadratic variations in
physical space, we detect the first nonvanishing high-frequency spectral
mismatch via localized Fourier coefficients and use a normalized Whittle score
to identify parameters. More broadly, the localized
spectral probing framework used here for detecting subtle covariance differences in
Gaussian random fields may be useful for studying identifiability and
estimation in other spatial models.

\end{abstract}
\maketitle

\section{Introduction}

Gaussian random fields are a central tool for modeling spatial dependence
\cite{stein1999,banerjee2003}.
Let
\[
Y=\{Y(t):t\in\R^d\}
\]
be a mean-zero stationary Gaussian field with covariance
\[
\mathrm{Cov}(Y(s),Y(t)) = K(t-s).
\]
Many spatial models further assume isotropy, meaning that
there exists a function
$\varphi:[0,\infty)\to\R$ such that
\[
K(h) = \varphi(|h|), \qquad h\in\R^d .
\]
Among stationary isotropic covariance families, the Mat\'ern class occupies a distinguished position.

For variance parameter $\sigma^2>0$, range parameter $\alpha>0$, and
smoothness parameter $\nu>0$, the Mat\'ern covariance function is \cite{stein1999}
\begin{equation}\label{eq:intro-matern}
K(h)=
\sigma^2
\frac{(\alpha |h|)^\nu}{2^{\nu-1}\Gamma(\nu)}
\mathcal{K}_\nu(\alpha |h|),
\qquad h\in\R^d,
\end{equation}
where $\mathcal{K}_\nu$ is the modified Bessel function of the second kind,
and $K(0)=\sigma^2$ by continuity.
The parameter $\sigma^2$ sets the marginal variance, $\alpha$ governs the
spatial scale of dependence, and $\nu$ controls regularity.
Larger values of $\nu$ correspond to smoother fields; in particular, the
Mat\'ern model provides a continuous hierarchy ranging from rough fields
(for example the exponential model when $\nu=\tfrac12$) to very smooth
ones, and after suitable rescaling it approaches the squared exponential
kernel as $\nu\to\infty$. The survey \cite{porcu2024matern} gives a  panoramic overview of the multifaceted use of the Mat\'ern family in statistical
applications and  beyond. 

For stationary fields it is often more convenient to work in the spectral
domain.
By Bochner's theorem,
\[
K(h)=\frac{1}{(2\pi)^d}\int_{\R^d} e^{i\xi\cdot h}f(\xi)\,d\xi,
\]
where $f$ is the spectral density.
For the Mat\'ern covariance \eqref{eq:intro-matern}, the spectral density is \cite[p. 49]{stein1999}
\begin{equation}\label{eq:intro-spectral}
f_{\sigma,\alpha,\nu}(\xi)
=
C_{d,\nu}\sigma^2\alpha^{2\nu}
(\alpha^2+|\xi|^2)^{-(\nu+d/2)},
\qquad \xi\in\R^d,
\end{equation}
for a constant $C_{d,\nu}>0$ depending only on $d$ and $\nu$.
Hence, as $|\xi|\to\infty$,
\begin{equation}\label{eqn:highfreq}
f(\xi)\sim C\sigma^2\alpha^{2\nu}|\xi|^{-2\nu-d},
\qquad |\xi|\to\infty.
\end{equation}
This high-frequency tail is fundamental for fixed-domain asymptotics.
It shows, first, that $\nu$ determines the spectral decay and therefore the
local regularity of the field.
Second, it identifies the combination
\begin{equation} \label{eqn:microerg}
m = m(\sigma,\alpha):=\sigma^2\alpha^{2\nu}
\end{equation}
as the leading coefficient in the spectral tail.
In the parameterization used in this paper, this is the
\emph{microergodic parameter} \cite{stein1999}.

\subsection{Fixed-domain asymptotics and microergodicity}

In spatial statistics one often observes a random field on a fixed bounded
domain while the sampling design becomes increasingly dense.
This is the \emph{fixed-domain}, or \emph{infill}, asymptotic regime.
Since the data arise from a single realization of a Gaussian random field, an important issue for statistical inference is not merely whether different parameter values induce different Gaussian measures, but whether the corresponding measures are equivalent or mutually singular on the observation sigma-field. Under fixed-domain asymptotics, this distinction governs which covariance features are consistently estimable from dense observations of a single path.

The present paper studies this issue for stationary Mat\'ern fields on
Euclidean domains and asks which covariance parameters are identifiable,
and therefore estimable, in the fixed-domain asymptotic
regime. For the Mat\'ern family with known smoothness $\nu$, the leading spectral
asymptotics already suggest what should happen.
Two models with the same value of
\[
m = m(\sigma,\alpha) :=\sigma^2\alpha^{2\nu}
\]
have the same high-frequency tail to first order.
This makes $m(\sigma,\alpha)$ the natural estimable combination under infill
asymptotics, while the separate identifiability of $\sigma^2$ and $\alpha$
depends on whether the next-order information in the covariance or the
spectrum can also be recovered from dense data.

A seminal result of \citet{Zhang2004} shows that when $d\le 3$, two
stationary Mat\'ern fields with the same smoothness $\nu$ and the same
microergodic parameter $m$ induce \emph{equivalent} Gaussian measures on every bounded
domain.
In these dimensions the variance and range parameters cannot be separated
consistently from a single dense realization.
At the opposite end, \citet{Anderes2010} proved that when $d>4$, the
corresponding Mat\'ern measures are orthogonal whenever the parameters
differ.
Equivalently, variance and range can be consistently separated in that
regime. For stationary Mat\'ern fields on $\R^d$, the $d=4$ case has remained open until now.

Dimension four turns out to be the critical case. When two Mat\'ern models
have the same microergodic parameter $m$ and smoothness $\nu$ but different
range parameters $\alpha$, the discrepancy between their covariance
structures lies exactly at the boundary between summable and non-summable
second–order effects across frequencies. In lower dimensions the cumulative
discrepancy remains bounded, while in higher dimensions it grows at a
polynomial rate. In dimension four the accumulation is much slower: the
distinguishing signal grows only logarithmically as higher frequencies are
incorporated. As we show below, this slow but unbounded accumulation is
precisely what makes the four–dimensional case identifiable.
Recent work on related Gaussian field models has continued to identify
dimension four as the delicate boundary case in the Euclidean setting
\citep{porcu2024matern, BoliKirc2023, li2023inference}.

\subsection{Our contribution}

We settle the remaining Euclidean case for $d = 4$.
We prove in Theorem \ref{thm:main} that in dimension $d=4$, two stationary Mat\'ern fields with the
same smoothness parameter and the same microergodic parameter, but with
different range parameters, induce mutually singular Gaussian measures on
every countable dense observation set in a bounded domain.
We prove this by first showing that the associated laws on path space are mutually singular and then projecting to dense sets.

More precisely, let $Y_1$ and $Y_2$ be stationary Mat\'ern fields on $\R^4$
with parameters $(\sigma_1,\alpha_1)$ and $(\sigma_2,\alpha_2)$ and common
smoothness $\nu>0$.
Assume
\begin{equation} \label{eqn:merg}
\sigma_1^2\alpha_1^{2\nu}=\sigma_2^2\alpha_2^{2\nu},
\qquad
\alpha_1\neq\alpha_2.
\end{equation}
Then Theorem~\ref{thm:main} shows that the induced Gaussian measures on any
countable dense observation set are mutually singular.
Thus, unlike the case $d\le 3$, the range parameter is asymptotically
distinguishable from dense fixed-domain data for $d = 4$.

\subsection{Idea of the proof}
Our key idea is to probe the field using localized Fourier coefficients, revealing a small variance mismatch at high frequencies whose accumulation produces a logarithmic signal in dimension four.
Our argument works for any $d \geq 4$ and thus also recovers the orthogonality results for $d > 4$ from \citet{Anderes2010}.

Let $D \subset \R^4$ be a bounded domain and suppose we have two Mat\'ern laws $\mathbf{P}_1$ and $\mathbf{P}_2$ satisfying \eqref{eqn:merg} on $D$. 
Our proof resolves the critical case $d=4$ by working in the spectral domain.
 Writing $p:=\nu+2$, the Mat\'ern spectral density in dimension four is
\[
f_{\sigma,\alpha,\nu}(\xi)
=
C_\nu\,\sigma^2\alpha^{2\nu}(\alpha^2+|\xi|^2)^{-p}.
\]
Under the microergodic matching condition \eqref{eqn:merg}
\[
\sigma_1^2\alpha_1^{2\nu}=\sigma_2^2\alpha_2^{2\nu},
\qquad
\alpha_1\neq\alpha_2,
\]
the two spectral densities agree at leading order,
\[
f_j(\xi)\sim C_\nu m\,|\xi|^{-2p},
\qquad |\xi|\to\infty,
\]
with $m=\sigma_j^2\alpha_j^{2\nu}$.
The first discrepancy appears at the next order,
\[
f_2(\xi)-f_1(\xi)\asymp |\xi|^{-2p-2},
\qquad
\frac{f_2(\xi)-f_1(\xi)}{f_1(\xi)}\asymp |\xi|^{-2}.
\]
In dimension four this decay is critical, since
\[
\sum_{|k|\le N}|k|^{-4}\asymp \log N .
\]
Thus the mismatch accumulates only logarithmically across frequencies, and
this marginal divergence is the signal exploited in our construction.

To access this signal, we ``probe'' the field by constructing localized Fourier coefficients. Fix a smooth cutoff $\chi\in C_c^\infty(D)$\footnote{
For example, if $\overline{B(0,R)}\subset D$, one may take
\[
\chi(t)=
\begin{cases}
\exp\!\left(-\dfrac{1}{1-\lvert t\rvert^2/R^2}\right), & \lvert t\rvert<R,\\[1ex]
0, & \lvert t\rvert\ge R.
\end{cases}
\]
} and define 
\begin{equation}\label{eq:Xk-def}
X_k := \int_D \chi(t)e^{-ik\cdot t}Y(t)\,dt,
\qquad k\in\N^4.
\end{equation}
Under the $j$th model, $X_k$ is a centered complex Gaussian random variable
with variance
\[
v_j(k)=\E_j|X_k|^2
=
(2\pi)^{-4}(f_j*|\widehat{\chi}|^2)(k).
\]
Thus $v_j(k)$ is a \emph{local smoothing} of $f_j$ over a frequency window of scale
$(\mathrm{diam}(\mathrm{supp}\,\chi))^{-1}$.

Define the relative variance mismatch
\begin{equation}\label{eqn:dk}
\delta_k=\frac{v_2(k)}{v_1(k)}-1.
\end{equation}
A direct expansion (Lemma~\ref{lem:diagonal-asymptotics}) shows that
\[
\delta_k \sim p(\alpha_1^2-\alpha_2^2)|k|^{-2}.
\]
Consequently, for the shell
\begin{equation} \label{eqn:Lambda}
\Lambda_N:=\{k\in\mathbb N^4:K_0\le \|k\|_\infty\le N\},
\qquad
L_N:=\sum_{k\in\Lambda_N}\delta_k^2,
\end{equation}
we have
\[
L_N\asymp \log N.
\]
The quantity $L_N$ turns out to be the natural information scale of the problem.

To separate models using the above observation, we consider the normalized quadratic statistic
\begin{equation}\label{eqn:quadstat}
T_N
=
\frac{1}{L_N}
\sum_{k\in\Lambda_N}
\delta_k
\left(
\frac{|X_k|^2}{v_1(k)}-1
\right).
\end{equation}
Each term is centered under $\mathbf P_1$, while under $\mathbf P_2$ it has mean $\delta_k$, so that
\[
\E_1[T_N]=0,
\qquad
\E_2[T_N]=1.
\]
Thus $T_N$ aggregates a sequence of weak frequency-wise signals into an
order-one separation. In Section~\ref{sec:comparison} we show that the statistic $T_N$ can be interpreted as a natural score-type statistic associated with the difference between the two localized spectral laws. 

A key point is that the variables $\{X_k\}$ are not independent. However, since the Fourier transform $\widehat{\chi}$ decays rapidly, localization yields that the covariances
\[
\mathbb{E}_j[X_k\overline{X_\ell}],
\qquad
\mathbb{E}_j[X_k X_\ell]
\]
decay rapidly in $|k-\ell|$ and $|k+\ell|$, respectively, for $j=1,2$. This off-diagonal decay enables us to control the variance of $T_N$ and show that: 
\[
\mathrm{Var}_j(T_N)\lesssim L_N^{-1}.
\]
Since $L_N\asymp\log N$, the variance vanishes along a sparse subsequence $\{N_s\}$,
and we obtain almost sure separation
\[
T_{N_s}\to0\quad \mathbf P_1\text{-a.s.},\qquad
T_{N_s}\to1\quad \mathbf P_2\text{-a.s.}
\]
This yields a separating event and hence mutual singularity. 

This viewpoint may be interpreted as a spectral analogue of the increment
method of \citet{Anderes2010}; see Section~\ref{sec:comparison} for further
details. In \citet{Anderes2010}, the second irregular covariance coefficient
is isolated in physical space, whereas here we isolate the first nonvanishing
spectral mismatch at high frequency. The critical dimension $d=4$ appears
because the information scale $L_N$ diverges only logarithmically.
More broadly, this spectral probing perspective provides a systematic way to analyze equivalence and orthogonality of 
Gaussian measures via localized frequency information. It also points toward a new class of inference procedures, in which parameters are identified by aggregating weak high-frequency discrepancies, in the spirit of localized Whittle estimation.

Our construction is related to, but not identical with, tapering. Classical lag tapering \cite{zhang2008covariance} \emph{replaces} the stationary covariance kernel \(K(h)\) by
\[
K_{\mathrm{tap}}(h)=K(h){\omega}_\chi(h), \qquad {\omega}_\chi(h):=\int_{\mathbb R^4}\chi(u+h)\chi(u)\,du,
\]
which corresponds to replacing the spectral density \(f\) by \((2\pi)^{-4}(f*|\widehat{\chi}|^2)\); we do not change the covariance kernel. 
Thus, although at each frequency we recover the same smoothed spectral quantity as in tapering, the off-diagonal structure induced by our spectral localization does not coincide with that of any tapered stationary covariance model.

\subsection{Related Work}

A recent paper of Bolin and Kirchner \cite{BoliKirc2023} studies equivalence of
Gaussian measures for generalized Whittle--Mat\'ern fields on a bounded domain
$D\subset \R^d$, defined through fractional elliptic operators with homogeneous
Dirichlet boundary conditions. In the classical constant-coefficient case, their
model is
\[
\mu = N(m,\tau^{-2}L^{-2\beta}),
\qquad
L=-\Delta+\kappa^2,
\]
viewed as a Gaussian measure on $L^2(D)$. Here $\beta$ plays the role of the
smoothness index, with the usual whole-space correspondence
\[
\nu = 2\beta-\frac d2,
\]
and $\kappa$ plays the role of the range parameter. In the constant-coefficient
case, the scaling parameter $\tau^{-2}$ is proportional, up to a constant
depending only on $(d,\nu)$, to the Mat\'ern microergodic combination
$\sigma^2\kappa^{2\nu}$. Thus their parameter $\tau$ is the bounded-domain
analogue of the leading spectral scale in \eqref{eqn:highfreq}.

Their Corollary~3.3 gives a complete equivalence classification for the classical
bounded-domain Whittle--Mat\'ern model. In dimensions $d\le 3$, equivalence
holds if and only if
\[
\beta=\tilde\beta,
\qquad
\tau=\tilde\tau,
\]
together with the corresponding Cameron--Martin condition on the means. In
dimensions $d\ge 4$, equivalence additionally requires
$
\kappa^2=\tilde\kappa^2.
$
Therefore, in the bounded-domain classical Whittle--Mat\'ern model, matching the
leading scale is no longer sufficient in dimension $4$: once $\beta$ and $\tau$
are matched, any nontrivial change in $\kappa$ forces non-equivalence. This mirrors our result for 
the stationary Euclidean model obtained in this paper.

We study the restriction to a bounded observation domain of a stationary
Mat\'ern field defined on all of $\R^4$. Thus translation invariance holds in
the ambient Euclidean space and is broken only by the restriction of the
observations to $D$. By contrast, the Whittle--Mat\'ern field in
\cite{BoliKirc2023} is constructed intrinsically on $D$ through the Dirichlet
operator $L=-\Delta+\kappa^2$, so the boundary conditions are part of the model
itself and the resulting field is generally non-stationary. Moreover, on a
bounded domain the classical stationary Mat\'ern model does not admit a fixed
fractional-operator representation since its eigenfunctions depend on the smoothness
parameter. Thus the operator-based framework of \cite{BoliKirc2023} does not
directly apply. In the limit as the domain expands to $\R^d$, the model converges to the stationary Mat\'ern field, 
and the same dimension-dependent equivalence behavior is expected.

At the level of asymptotics, however, the mechanisms behind the two proofs are
strikingly similar. In the constant-coefficient bounded-domain model, the
Dirichlet eigenfunctions \emph{diagonalize} both covariance operators. If
$\hat\lambda_j$ denotes the $j$th eigenvalue of the Dirichlet Laplacian on $D$,
then
\[
\lambda_j=\hat\lambda_j+\kappa^2,
\qquad
\tilde\lambda_j=\hat\lambda_j+\tilde\kappa^2,
\]
and the Feldman--H\'ajek criterion \cite[Theorem 2.5]{DPZ2014} reduces to the square-summability of the
diagonal ratios
\[
c_j-1,
\qquad
c_j=
\Bigl(\frac{\tilde\tau}{\tau}\Bigr)^{1/\beta}
\frac{\hat\lambda_j+\tilde\kappa^2}{\hat\lambda_j+\kappa^2}.
\]
After matching the leading scale, one has
\[
c_j-1 \sim (\tilde\kappa^2-\kappa^2)\hat\lambda_j^{-1}.
\]
Since Weyl's law gives $\hat\lambda_j\asymp j^{2/d}$, it follows that
$
(c_j-1)^2 \asymp j^{-4/d}
$
and for $d=4$ the corresponding partial sums satisfy
\[
\sum_{j\le N}(c_j-1)^2 \asymp \sum_{j\le N}\frac1j
= \log N + O(1).
\]
Thus $d=4$ is exactly the logarithmic borderline for the
Feldman--H\'ajek square-summability condition in this setting.

Our proof isolates the same critical second-order mismatch in continuous
frequency. The main difference lies in the proof architecture. In \cite{BoliKirc2023}, the
constant-coefficient bounded-domain model is \textit{exactly diagonal} in a common
Dirichlet eigenbasis, so equivalence can be read off from a diagonal
square-summability condition. In our stationary setting, we must control the off-diagonal covariances and
pseudo-covariances introduced by localization. This is our key technical contribution. The connection with Feldman--H\'ajek is through the quantity \(L_N\): in an
exactly diagonal model, \(L_N\) is the truncated Feldman--H\'ajek
square-summability sum, while in our localized setting it becomes the leading
score/information scale on increasing finite-dimensional frequency windows.
Our off-diagonal estimates show that this same scale continues to govern the
true model.
\subsection{Organization of the paper}

Section~2 states the main theorem precisely. Section 3 places
our work in the context of previous work and gives a likelihood
interpretation of our statistic \eqref{eqn:quadstat}.
Section~4 passes from dense observations to path-space laws on
$C(\overline D)$. 
Sections~5--8 develop the localized Fourier coefficients, establish diagonal
asymptotics and off-diagonal decay, and construct the separating statistic; 
Section 9 gives the proof of the main result. Section 10 gives a simulation study showing that an estimation procedure based on the
Whittle pseudo-likelihood recovers the parameter $\alpha$.
Throughout the paper, we write \(a \gtrsim b\) if there exists a constant \(C>0\), depending only on fixed model parameters (such as \(d\), \(\nu\), the domain \(D\), and the cutoff \(\chi\)), such that \(a \ge C\, b\); we write \(a \asymp b\) if both \(a \gtrsim b\) and \(b \gtrsim a\). We use \(O(\cdot)\) and \(o(\cdot)\) in the usual asymptotic sense as \(|k|\to\infty\) (or \(N\to\infty\)), with all implicit constants depending only on these fixed parameters.

\section{Main Result}
Fix $\nu>0$ and write
\[
p:=\nu+2.
\]
Let $Y_j=\{Y_j(t):t\in\R^4\}$, $j=1,2$, be mean-zero stationary Mat\'ern Gaussian fields with parameters
$(\sigma_j,\alpha_j)$ and common smoothness $\nu$. Their spectral densities are
\begin{equation}\label{eq:matern-spectral}
f_j(\xi)=C_\nu\,\sigma_j^2\alpha_j^{2\nu}\,(\alpha_j^2+\abs{\xi}^2)^{-p},
\qquad \xi\in\R^4,
\end{equation}
where $C_\nu>0$ depends only on $\nu$ and on the Fourier convention.

We impose the microergodic matching condition
\begin{equation}\label{eq:micro-match}
\sigma_1^2\alpha_1^{2\nu}=\sigma_2^2\alpha_2^{2\nu}=:m,
\qquad
\alpha_1\neq \alpha_2.
\end{equation}
Under \eqref{eq:micro-match},
\[
f_j(\xi)=C_\nu m\,(\alpha_j^2+\abs{\xi}^2)^{-p}.
\]

The following is our main result.

\begin{theorem}\label{thm:main}
Let $D \subset \R^4$ be a bounded domain with nonempty interior and $S \subset D$ be countable and dense. Denote by $\mathbf{P}_j^S$ the law of the coordinate process
$\{Y_j(t):t\in S\}$ on $\R^S$. 
Assume \eqref{eq:micro-match}. Then
\[
\mathbf{P}_1^S\perp \mathbf{P}_2^S.
\]
In particular, for the dense grid
\[
D_\infty:=\bigcup_{n\ge1} D\cap n^{-1}\Z^4,
\]
the laws of the coordinate processes \(\{Y_j(t):t\in D_\infty\}\) are mutually singular.
\end{theorem}

The proof proceeds by first establishing singularity of the induced path laws on
\(C(\overline D)\) and then transferring that conclusion to dense observation sets.

\section{The critical dimension, Anderes' method, and our approach} \label{sec:comparison}

The mechanism behind the phase transition is already visible in the local
covariance expansion of the Mat\'ern model; see \citet{Anderes2010}.
Near the origin, the covariance admits a sequence of non-polynomial terms with leading orders $|h|^{2\nu}$ and $|h|^{2\nu+2}$ (with logarithmic modifications), whose coefficients are proportional to $\sigma^2 \alpha^{2\nu}$ and $\sigma^2 \alpha^{2\nu+2}$, respectively.
The first coefficient is exactly the microergodic parameter
\(\sigma^2\alpha^{2\nu}\).
Thus, if only the leading irregular term can be recovered from the data,
then one can estimate only the microergodic combination.
If the next irregular term can also be identified, then the parameters
\(\sigma^2\) and \(\alpha\) can be separated.

\citet{Anderes2010} makes this precise using higher-order directional
increments. For a fixed nonzero direction \(h\), define
\[
\Delta_h Y(t)=Y(t+h)-Y(t),
\qquad
\Delta_h^\ell Y(t)=\Delta_h(\Delta_h^{\ell-1}Y(t)).
\]
For observations on the grid
\[
\Omega_n=\Omega\cap n^{-1}\mathbb Z^d,
\]
he considers the quadratic variation statistic
\[
Q_n^\ell(h)
=
\frac{1}{\#\Omega_n}\sum_{j\in\Omega_n}
n^{2\nu}\bigl(\Delta_{h/n}^\ell Y(j)\bigr)^2 .
\]
Lemma~1 of \citet{Anderes2010} shows that for \(\ell>\nu+1\),
\[
\E\bigl(\Delta_{h/n}^\ell Y(t)\bigr)^2
=
a_\nu^\ell(h)n^{-2\nu}
+
b_\nu^\ell(h)n^{-2\nu-2}
+
o(n^{-2\nu-2}),
\]
where \(a_\nu^\ell(h)\) is proportional to \(\sigma^2\alpha^{2\nu}\) and
\(b_\nu^\ell(h)\) is proportional to \(\sigma^2\alpha^{2\nu+2}\).
Thus \(Q_n^\ell(h)\) consistently estimates the leading coefficient
\(a_\nu^\ell(h)\), whose dependence on the unknown parameters occurs
through the microergodic combination \(\sigma^2\alpha^{2\nu}\).

To isolate the next-order coefficient, Anderes compares two increment
orders \(u\neq v\) and considers
\[
R_n^{u,v}(h)
=
n^2\!\left(
Q_n^u(h)
-
\frac{a_\nu^u(h)}{a_\nu^v(h)}Q_n^v(h)
\right).
\]
Theorem~2 of \citet{Anderes2010} shows that, under the stated dimension
conditions, this statistic converges almost surely to
\[
b_\nu^u(h)-\frac{a_\nu^u(h)}{a_\nu^v(h)}\,b_\nu^v(h).
\]
The normalization removes the leading contribution associated with the
coefficients \(a_\nu^u(h)\) and \(a_\nu^v(h)\), leaving a limit determined
by the second-order coefficients \(b_\nu^u(h)\) and \(b_\nu^v(h)\).

Lemma~2 of \citet{Anderes2010} shows that one can choose \(u\) and \(v\)
so that this combination is nonzero.
Since each \(b_\nu^u(h)\) is proportional to \(\sigma^2\alpha^{2\nu+2}\),
the statistic therefore recovers information at the second irregular
order.
Together with the first-order coefficient \(a_\nu^u(h)\), this allows
\(\sigma^2\) and \(\alpha\) to be separated when \(d>4\).

\subsection{Why \(d=4\) is the critical threshold}

A key feature of the construction of \citet{Anderes2010} used to recover
\(\sigma^2\) and \(\alpha\) is that the convergence of the second statistic
requires the dimension condition \(d>4\).
The issue is variance rather than bias.
Theorem~1 of \citet{Anderes2010} shows that
\[
\operatorname{Var}(Q_n^\ell)
\lesssim
\begin{cases}
n^{4(\nu-\ell)}, & 4(\nu-\ell)>-d,\\[4pt]
n^{-d}\log n, & 4(\nu-\ell)=-d,\\[4pt]
n^{-d}, & 4(\nu-\ell)<-d.
\end{cases}
\]
Hence the first coefficient can be recovered once the increment order $\ell$ is high enough.
The second coefficient is harder: the statistic \(R_n^{u,v}(h)\) carries an extra
factor \(n^2\), because the second irregular term is smaller by two powers of \(n\).
To make this renormalized statistic concentrate almost surely, Anderes requires
\[
4<\min\{2u-2\nu,d\},
\qquad
4<\min\{2v-2\nu,d\}.
\]
The increment orders \(u,v\) can be chosen large, but the condition \(d>4\) remains.
This is the source of the phase transition. Below four dimensions, the second
coefficient is lost in the fluctuations, and above four dimensions, it can be recovered
from dense lattice data. Dimension four is exactly the point at which the second-order
signal is still present but accumulates only at the logarithmic scale.

\subsection{Score interpretation and likelihood-ratio expansion} \label{sec:infscale}

It is useful to interpret the statistic $T_N$ in \eqref{eqn:quadstat}
in likelihood-theoretic terms. In particular, the statistic $T_N$
is essentially the Gaussian score in the $\alpha^2$ direction.
Recall the localized Fourier coefficients
\[
X_k(g)
=
\int_D \chi(t)e^{-ik\cdot t}g(t)\,dt,
\qquad k\in\N^4 .
\]
Under a Mat\'ern model with parameters $(m,\alpha)$ we write
\[
v(k;\alpha)=\E_\alpha |X_k|^2 .
\]
Lemma~\ref{lem:diagonal-asymptotics} shows that
\[
v(k;\alpha)=c_\chi f_\alpha(k)(1+o(1)),
\]
where
\[
f_\alpha(\xi)
=
C_\nu m(\alpha^2+|\xi|^2)^{-p},
\qquad p=\nu+2 .
\]
Hence the derivative $\partial_{\alpha^2}$ taken along the
microergodic curve (\textit{i.e.}, at fixed $m$) yields
\[
\partial_{\alpha^2}\log v(k;\alpha)
\sim
-\frac{p}{\alpha^2+|k|^2}
=
-p|k|^{-2}+o(|k|^{-2}) .
\]

Fix two Mat\'ern models \((m,\alpha_1)\) and \((m,\alpha_2)\), and write
\[
v_j(k):=v(k;\alpha_j), \qquad j=1,2.
\]
Recall the frequency mismatch from \eqref{eqn:dk},
\[
\delta_k:=\frac{v_2(k)}{v_1(k)}-1.
\]
By Lemma~\ref{lem:diagonal-asymptotics},
\[
\delta_k
=
p(\alpha_1^2-\alpha_2^2)|k|^{-2}+o(|k|^{-2}),
\qquad p=\nu+2.
\]
Thus \(\delta_k\) coincides, to first order, with the finite-difference approximation
\[
(\alpha_2^2-\alpha_1^2)\,\partial_{\alpha^2}\log v(k;\alpha_1),
\]
and is therefore proportional to the score weight \(\partial_{\alpha^2}\log v(k;\alpha_1)\).
To build intuition, we first analyze the corresponding diagonal
approximation.

Fix \(m>0\) and let \(\alpha_1\neq \alpha_2\). For \(j=1,2\), let
\(\mathbf P_j^{\mathrm{diag}}\) denote the diagonal approximation to the law of
\(\{X_k\}_{k\in\Lambda_N}\) under the Mat\'ern model with parameters
\((m,\alpha_j)\), that is, the law under which the variables
\(\{X_k\}_{k\in\Lambda_N}\) are independent centered complex Gaussian with
variances
$
v_j(k):=\E_j|X_k|^2 .
$
Define
\[
S_N:=\sum_{k\in\Lambda_N}\delta_k
\left(\frac{|X_k|^2}{v_1(k)}-1\right),
\qquad
L_N:=\sum_{k\in\Lambda_N}\delta_k^2.
\]
We show in Lemma \ref{lem:matdiag} below that
\begin{equation}\label{eq:matern-diagonal-llr}
\log\frac{d\mathbf P_2^{\mathrm{diag}}}{d\mathbf P_1^{\mathrm{diag}}}
\approx
S_N-\frac12L_N.
\end{equation}
Thus $S_N$ can be interpreted as the score and $L_N = \sum_{k\in\Lambda_N}\delta_k^2$ is the 
natural information scale. 
The quadratic statistic $T_N$ in \eqref{eqn:quadstat} is just the normalized score:
\begin{equation}\label{eq:TN-score}
T_N=\frac{S_N}{L_N}.
\end{equation}
Thus \(T_N\) is the natural score-type statistic associated with the
difference between the two localized spectral laws. 
In the diagonal approximation, Lemma \ref{lem:matdiag} shows that \(T_N\) has the expected separating behavior:
it converges in probability to \(0\) under the Mat\'ern model with parameter \(\alpha_1\)
and to \(1\) under the Mat\'ern model with parameter \(\alpha_2\).  Our proof for the orthogonality in the $d=4$ case exploits the fact that the true localized covariance structure, although not exactly diagonal,
preserves this same leading score/information mechanism.
\begin{lemma}[Diagonal likelihood expansion for the localized Mat\'ern coefficients]
\label{lem:matdiag}
Fix \(m>0\) and let \(\alpha_1\neq \alpha_2\). For \(j=1,2\), let \(\mathbf P_j^{\mathrm{diag}}\) denote the diagonal approximation
to the law of \(\{X_k\}_{k\in\Lambda_N}\) under the Mat\'ern model with parameters
\((m,\alpha_j)\), that is, the law under which the variables
\(\{X_k\}_{k\in\Lambda_N}\) are independent, centered, circular complex Gaussian with
\[
\E_j^{\mathrm{diag}} |X_k|^2 = v_j(k), \qquad k\in\Lambda_N.
\]
Then
\begin{equation}\label{eq:matern-diagonal-llr}
\log\frac{d\mathbf P_2^{\mathrm{diag}}}{d\mathbf P_1^{\mathrm{diag}}}
=
S_N-\frac12L_N+R_N
\end{equation}
where
\[
R_N=o_{\mathbf P_j^{\mathrm{diag}}}(L_N), \qquad j=1,2.
\]

We have,
\[
\E_1^{\mathrm{diag}}[S_N]=0,
\qquad
\Var_1^{\mathrm{diag}}(S_N)=L_N,
\]
and
\[
\E_2^{\mathrm{diag}}[S_N]=L_N,
\qquad
\Var_2^{\mathrm{diag}}(S_N)
=
\sum_{k\in\Lambda_N}\delta_k^2(1+\delta_k)^2 \sim L_N.
\]
Moreover
\[
T_N=\frac{S_N}{L_N}\to 0
\quad\text{in \(\mathbf P_1^{\mathrm{diag}}\)-probability},
\qquad
T_N=\frac{S_N}{L_N}\to 1
\quad\text{in \(\mathbf P_2^{\mathrm{diag}}\)-probability}.
\]
\end{lemma}

\begin{proof}
For each \(j=1,2\), the diagonal approximation treats the coefficients
\(\{X_k\}_{k\in\Lambda_N}\) as independent centered circular complex Gaussian
variables with variances \(v_j(k)\). Since the density of a centered circular
complex Gaussian variable with variance \(v\) is
\[
\varphi_v(z)=\frac1{\pi v}\exp\!\left(-\frac{|z|^2}{v}\right), \qquad z\in\mathbb C,
\]
we obtain
\[
\log\frac{d\mathbf P_2^{\mathrm{diag}}}{d\mathbf P_1^{\mathrm{diag}}}
=
\sum_{k\in\Lambda_N}
\left[
-\log\frac{v_2(k)}{v_1(k)}
+
|X_k|^2\left(\frac1{v_1(k)}-\frac1{v_2(k)}\right)
\right].
\]
Using \(v_2(k)=v_1(k)(1+\delta_k)\), this becomes
\[
\log\frac{d\mathbf P_2^{\mathrm{diag}}}{d\mathbf P_1^{\mathrm{diag}}}
=
\sum_{k\in\Lambda_N}
\left[
-\log(1+\delta_k)
+
\frac{|X_k|^2}{v_1(k)}\frac{\delta_k}{1+\delta_k}
\right].
\]
For \(|\delta|\) small,
\[
\log(1+\delta)=\delta-\frac12\delta^2+O(\delta^3),
\qquad
\frac{\delta}{1+\delta}=\delta-\delta^2+O(\delta^3),
\]
hence
\begin{align*}
\log\frac{d\mathbf P_2^{\mathrm{diag}}}{d\mathbf P_1^{\mathrm{diag}}}
&=
\sum_{k\in\Lambda_N}
\delta_k\left(\frac{|X_k|^2}{v_1(k)}-1\right)
-\frac12\sum_{k\in\Lambda_N}\delta_k^2 \\
&\quad
-\sum_{k\in\Lambda_N}\delta_k^2
\left(\frac{|X_k|^2}{v_1(k)}-1\right)
+
E_N,
\end{align*}
where
\[
|E_N|
\le
C\sum_{k\in\Lambda_N}
\left(1+\frac{|X_k|^2}{v_1(k)}\right)|\delta_k|^3.
\]
Thus \eqref{eq:matern-diagonal-llr} holds with
\[
R_N
:=
-\sum_{k\in\Lambda_N}\delta_k^2
\left(\frac{|X_k|^2}{v_1(k)}-1\right)
+
E_N.
\]
We next show that \(R_N=o_{\mathbf P_j^{\mathrm{diag}}}(L_N)\) for \(j=1,2\).
Set
\[
Y_k:=\frac{|X_k|^2}{v_1(k)}-1.
\]
Then
\[
R_N=-\sum_{k\in\Lambda_N}\delta_k^2Y_k+E_N.
\]

Under \(\mathbf P_1^{\mathrm{diag}}\), the variable \(X_k/v_1(k)^{1/2}\) is
standard centered circular complex Gaussian, so
\[
\E_1^{\mathrm{diag}}[Y_k]=0,
\qquad
\Var_1^{\mathrm{diag}}(Y_k)=1.
\]
By independence,
\[
\Var_1^{\mathrm{diag}}\!\left(\sum_{k\in\Lambda_N}\delta_k^2Y_k\right)
=
\sum_{k\in\Lambda_N}\delta_k^4.
\]
Under \(\mathbf P_2^{\mathrm{diag}}\),
\[
\E_2^{\mathrm{diag}}[Y_k]=\delta_k,
\qquad
\Var_2^{\mathrm{diag}}(Y_k)=(1+\delta_k)^2,
\]
and therefore
\[
\Var_2^{\mathrm{diag}}\!\left(\sum_{k\in\Lambda_N}\delta_k^2Y_k\right)
=
\sum_{k\in\Lambda_N}\delta_k^4(1+\delta_k)^2.
\]
By Lemma~\ref{lem:diagonal-asymptotics},
\[
\delta_k=p(\alpha_1^2-\alpha_2^2)|k|^{-2}+o(|k|^{-2}),
\]
so
\[
\sum_{k\in\Lambda_N}\delta_k^4
\asymp
\sum_{k\in\Lambda_N}|k|^{-8}
=O(1),
\qquad
L_N=\sum_{k\in\Lambda_N}\delta_k^2\asymp \log N.
\]
Hence
\[
\sum_{k\in\Lambda_N}\delta_k^2Y_k
=
O_{\mathbf P_j^{\mathrm{diag}}}(1),
\qquad j=1,2.
\]
For the cubic error term, again by \(\delta_k=O(|k|^{-2})\),
\[
\sum_{k\in\Lambda_N}|\delta_k|^3
\asymp
\sum_{k\in\Lambda_N}|k|^{-6}
=
O(1).
\]
Moreover, under each \(\mathbf P_j^{\mathrm{diag}}\), the variables
\(|X_k|^2/v_1(k)\) have uniformly bounded first moments, since
\[
\E_1^{\mathrm{diag}}\!\left[\frac{|X_k|^2}{v_1(k)}\right]=1,
\qquad
\E_2^{\mathrm{diag}}\!\left[\frac{|X_k|^2}{v_1(k)}\right]
=
\frac{v_2(k)}{v_1(k)}=1+\delta_k.
\]
Therefore
\[
\E_j^{\mathrm{diag}}|E_N|
\le
C\sum_{k\in\Lambda_N}
\left(1+\E_j^{\mathrm{diag}}\!\left[\frac{|X_k|^2}{v_1(k)}\right]\right)|\delta_k|^3
=O(1),
\qquad j=1,2,
\]
and hence
\[
E_N=O_{\mathbf P_j^{\mathrm{diag}}}(1),
\qquad j=1,2.
\]
Combining the two bounds gives
\[
R_N=O_{\mathbf P_j^{\mathrm{diag}}}(1),
\qquad j=1,2.
\]
Since \(L_N\asymp \log N\to\infty\), it follows that
\[
R_N=o_{\mathbf P_j^{\mathrm{diag}}}(L_N),
\qquad j=1,2.
\]
It remains to verify the stated moments of \(S_N\). Since
\[
S_N=\sum_{k\in\Lambda_N}\delta_kY_k,
\]
the above calculations give
\[
\E_1^{\mathrm{diag}}[S_N]=0,
\qquad
\Var_1^{\mathrm{diag}}(S_N)=\sum_{k\in\Lambda_N}\delta_k^2=L_N,
\]
and
\[
\E_2^{\mathrm{diag}}[S_N]
=
\sum_{k\in\Lambda_N}\delta_k\,\E_2^{\mathrm{diag}}[Y_k]
=
\sum_{k\in\Lambda_N}\delta_k^2
=
L_N,
\]
while
\[
\Var_2^{\mathrm{diag}}(S_N)
=
\sum_{k\in\Lambda_N}\delta_k^2\Var_2^{\mathrm{diag}}(Y_k)
=
\sum_{k\in\Lambda_N}\delta_k^2(1+\delta_k)^2.
\]

Finally,
\[
\Var_1^{\mathrm{diag}}\!\left(\frac{S_N}{L_N}\right)=\frac1{L_N}\to0,
\]
and
\[
\Var_2^{\mathrm{diag}}\!\left(\frac{S_N}{L_N}\right)
=
\frac{1}{L_N^2}\sum_{k\in\Lambda_N}\delta_k^2(1+\delta_k)^2
\to0,
\]
since \(\sum \delta_k^2(1+\delta_k)^2\asymp L_N\). Hence Chebyshev's inequality
yields
\[
T_N=\frac{S_N}{L_N}\to 0
\quad\text{in \(\mathbf P_1^{\mathrm{diag}}\)-probability},
\qquad
T_N=\frac{S_N}{L_N}\to 1
\quad\text{in \(\mathbf P_2^{\mathrm{diag}}\)-probability}
\]
and the proof is finished.
\end{proof}
\subsection{Simulation study}
In this section we present some simulation studies that explore the behavior of the
statistic $T_N$.
All simulations are carried out directly in frequency space in dimension \(d=4\). We fix
\[
\mathcal I_M=\{-M,\dots,M-1\}^4,\qquad \xi_n=h_\xi n,\qquad h_\xi=\frac1q,\quad q\in\mathbb N,
\]
so the simulated spectrum is truncated to the box \([-\Omega,\Omega]^4\) with \(\Omega=Mh_\xi\).
 The condition \(h_\xi=1/q\) ensures that integer frequencies lie on the lattice: \(k=\xi_{qk}\).
Let $\phi \in C_c^\infty([-1,1])$ be the smooth bump
\[
\phi(u)=\exp\!\Bigl(-\frac{1}{1-u^2}\Bigr)\mathbf 1_{\{|u|<1\}},
\]
and define the tensor-product localization function at scale $R>0$ by
\[
\chi_R(t)=\prod_{r=1}^4 \phi(t_r / R).
\]
Its Fourier transform is
\[
\widehat{\chi}_R(\xi)=R^4 \prod_{r=1}^4 \widehat{\phi}(R \xi_r).
\]
In the simulation, we work with the sampled kernel
\[
K(d)=\widehat{\chi}_R(h_\xi d),
\]
where $d \in \mathbb{Z}^4$ indexes the frequency lattice.

For model \(j=1,2\), the discrete spectral density is
\[
f_j(\xi)=m_j(\alpha_j^2+|\xi|^2)^{-(\nu+2)},\qquad m_j=\sigma_j^2\alpha_j^{2\nu}.
\]
In the microergodically matched experiments we impose \(m_1 = m_2 := m\).
Since the field is real-valued, we construct a Hermitian-symmetric complex Gaussian family \(\{G_n\}_{n\in\mathcal I_M}\) as follows: we first draw independent mean-zero complex Gaussian variables on a set of representatives of the pairs \(\{n,-n\}\), and then extend to all \(n\in\mathcal I_M\) by imposing
$
G_{-n} = \overline{G_n},$
with \(G_n\) taken to be real Gaussian on self-conjugate modes. We then define
\[
Z_n^{(j)} = h_\xi^2 \sqrt{f_j(\xi_n)}\, G_n,
\qquad n \in \mathcal I_M.
\]
The localized coefficients are then defined by the discrete convolution
\[
X_n^{(j)}=\sum_{r\in\mathcal I_M} K(n-r)\,Z_r^{(j)},\qquad 
n\in\mathcal I_M,
\]
with \(X_k^{(j)}:=X_{qk}^{(j)}\) whenever \(qk\in\mathcal I_M\). Thus \(X_k\) is obtained as a discretization of the spectral representation
\[
X_k
=
\int_{\mathbb R^4}
\widehat{\chi}(k-\xi)\,\sqrt{f_j(\xi)}\,W(d\xi).
\]
The test is evaluated on the positive-frequency shell
\[
\Lambda_{K_0,K_1}
=
\{k\in\mathbb N_0^4: K_0\le \|k\|_\infty \le K_1\},
\]
retaining only those \(k\) for which \(qk\in\mathcal I_M\). Hence the simulation has two cutoffs: the global spectral cutoff \([-\Omega,\Omega]^4\) coming from the grid, and the shell cutoff \([K_0,K_1]\) used in the statistic.

The variance normalization is computed on the frequency grid as
\[
v_j(n)=h_\xi^4\sum_{r\in\mathcal I_M} f_j(\xi_r)\,|K(n-r)|^2.
\]
For integer frequencies \(k\) such that \(qk\in\mathcal I_M\), we identify
\[
X_k := X_{qk},
\qquad
v_j(k) := v_j(qk),
\]
so that \(v_j(k)=\E[|X_k|^2]\). Both \(X^{(j)}\) and the variance functions are discrete convolutions on the frequency grid, and are therefore computed efficiently by FFT. 

The score statistic is
\[
S_N=\sum_{k\in\Lambda_{K_0,K_1}}
\delta_k\Bigl(\frac{|X_k|^2}{v_1(k)}-1\Bigr),
\qquad
\delta_k=\frac{v_2(k)}{v_1(k)}-1,
\]
with normalization
\[
L_N=\sum_{k\in\Lambda_{K_0,K_1}}\delta_k^2,
\qquad
T_N=\frac{S_N}{L_N}.
\]

\subsection{Simulation results}

We report Monte Carlo estimates of the normalized statistic \(T_N=S_N/L_N\) under two microergodically matched Mat\'ern models with \(m=1\) and \(\nu=1.5\). The simulation parameters are
\[
M=20,\quad h_\xi=0.5,\quad K_0=3,\quad K_1=9.
\]
We use taper radius $R=2$, and approximate the Fourier transform 
$\widehat{\phi}$ numerically using Simpson quadrature with $Q=400$ subintervals.
As seen from the tables below (with $200$ Monte Carlo iterations) and figures \ref{fig:alpha2} and \ref{fig:alpha12}, the test
statistic $T_N$ is able to separate the models.
\begin{figure}[ht]
\centering
\includegraphics[width=0.6\textwidth]{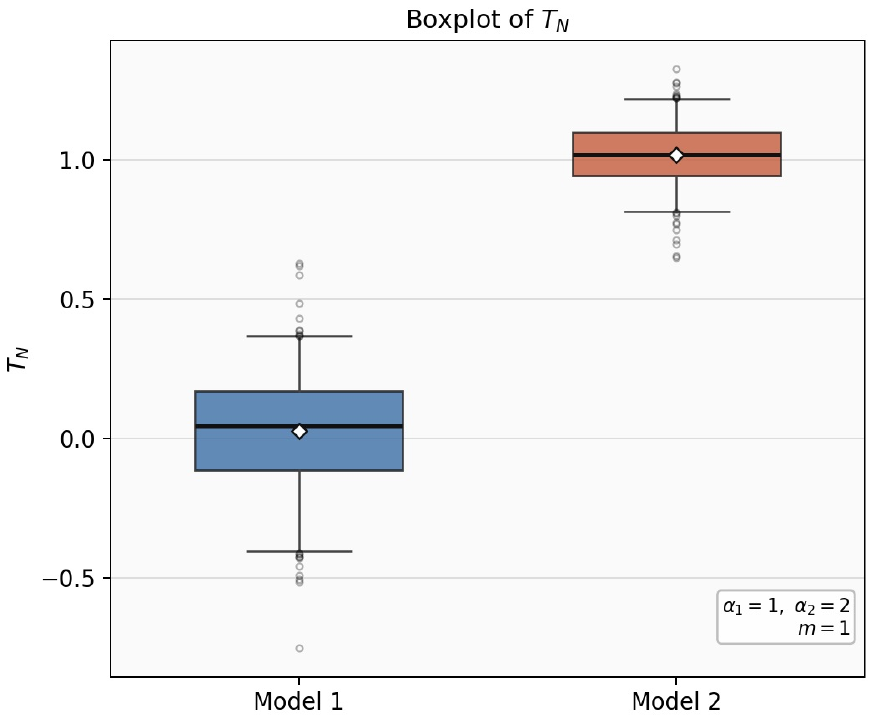}
\caption{Empirical distribution of \(T_N\) in Experiment 1 under the two microergodically matched models with \(\alpha_1=1\) and \(\alpha_2=2\). The statistic concentrates near \(0\) under model 1 and near \(1\) under model 2.}
\label{fig:alpha2}
\end{figure}

\begin{figure}[ht]
\centering
\includegraphics[width=0.6\textwidth]{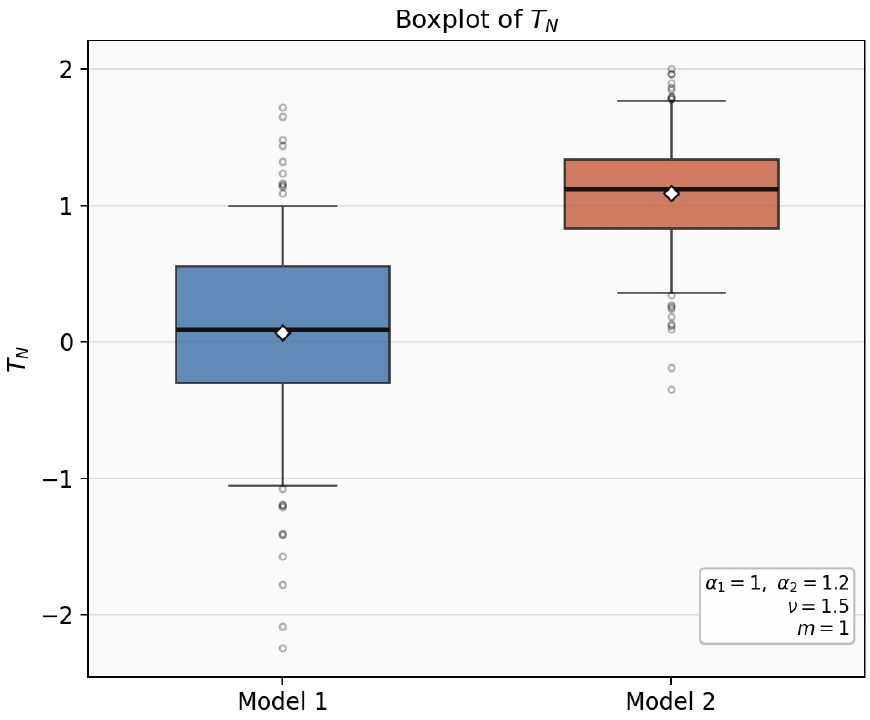}
\caption{Empirical distribution of \(T_N\) in Experiment 2 under the two microergodically matched models with \(\alpha_1=1\) and \(\alpha_2=1.2\). Separation persists but with increased variance due to weaker spectral separation.}
\label{fig:alpha12}
\end{figure}
\noindent\textbf{Experiment 1.} 
\[
(\alpha_1,\sigma_1)=(1,1),\qquad 
(\alpha_2,\sigma_2)=\bigl(2,\,2^{-3/2}\bigr),
\]
so that \(m_j=\sigma_j^2\alpha_j^{2\nu}=1\). The empirical mean and variance of \(T_N\) are:
\[
\begin{array}{c|cc}
\text{Model} & \E[T_N] & \Var(T_N) \\
\hline
1 & 0.026 & 0.052 \\
2 & 1.016 & 0.015
\end{array}
\]

\medskip

\noindent\textbf{Experiment 2.} 
\[
(\alpha_1,\sigma_1)=(1,1),\qquad 
(\alpha_2,\sigma_2)=\bigl(1.2,\,(1/1.2)^{3/2}\bigr),
\]
again with \(m_j=1\). The empirical mean and variance are:
\[
\begin{array}{c|cc}
\text{Model} & \E[T_N] & \Var(T_N) \\
\hline
1 & 0.067 & 0.456 \\
2 & 1.09 & 0.18
\end{array}
\] 
\section{Continuous modifications and dense observation sigma-fields}

We use the Fourier transform convention
\[
\what{g}(\xi):=\int_{\R^4} e^{-i\xi\cdot t} g(t)\,dt,
\qquad
g(t)=(2\pi)^{-4}\int_{\R^4} e^{i\xi\cdot t}\what g(\xi)\,d\xi.
\]

The following result is well-known, but we include it here for the reader's convenience.

\begin{lemma}\label{lem:continuous-modification}
Let $Y$ be a stationary Mat\'ern Gaussian random field on $\R^4$
with parameters $\sigma^2>0$, $\alpha>0$, and smoothness $\nu>0$.
Then $Y$ admits a modification with almost surely continuous
sample paths on $\overline D$.
\end{lemma}

\begin{proof}
Choose any $\beta\in(0,\min\{\nu,1\})$.
By stationarity and \eqref{eq:matern-spectral},
\[
\E|Y(t)-Y(s)|^2
=(2\pi)^{-4}\int_{\R^4}
|e^{i\xi\cdot t}-e^{i\xi\cdot s}|^2 f(\xi)\,d\xi .
\]
Since $0<\beta\le1$, there exists $C_\beta$ such that
\[
|e^{ix}-1|\le C_\beta |x|^\beta ,
\qquad x\in\R .
\]
Applying this with $x=\xi\cdot(t-s)$ gives
\[
|e^{i\xi\cdot t}-e^{i\xi\cdot s}|
=|e^{i\xi\cdot(t-s)}-1|
\le C_\beta |\xi|^\beta |t-s|^\beta .
\]
Substituting into the spectral representation yields
\[
\E|Y(t)-Y(s)|^2
\le C |t-s|^{2\beta}
\int_{\R^4} |\xi|^{2\beta} f(\xi)\,d\xi .
\]

Recall from \eqref{eqn:highfreq} that for $|\xi|\ge1$,
$
f(\xi)\asymp |\xi|^{-2\nu-4}$
so the integrand behaves like
\[
|\xi|^{2\beta}f(\xi)\asymp |\xi|^{2\beta-2\nu-4}.
\]

Passing to radial coordinates,
\[
\int_{|\xi|\ge1} |\xi|^{2\beta-2\nu-4}\,d\xi
= C\int_1^\infty r^{2\beta-2\nu-1}\,dr,
\]
which is finite precisely when $\beta<\nu$.

\medskip
Thus
\[
\E|Y(t)-Y(s)|^2 \le C|t-s|^{2\beta}.
\]

Since $Y(t)-Y(s)$ is Gaussian, for every $q\ge2$,
\[
\E|Y(t)-Y(s)|^q \le C_q |t-s|^{\beta q}.
\]
Choosing $q$ so large that $\beta q>4$, Kolmogorov's
continuity theorem yields a continuous modification on the compact set
$\overline D$.
\end{proof}

Let \(\widetilde Y_j\) denote a continuous modification of \(Y_j\) on the
compact set \(\overline D\) as implied by Lemma \ref{lem:continuous-modification} and
let  $P_j := \mathrm{Law}(\widetilde Y_j)$
denote its law on \(C(\overline D)\), equipped with the uniform norm.
We first show that, to prove the singularity of the laws of the original coordinate process $\{Y_j(t):t\in S\}$, it is enough to show that $\mathbf P_1\perp \mathbf P_2$ on $C(\overline D)$.

\begin{lemma}\label{lem:borel-generated}
Let $K$ be a compact metric space and let $S\subset K$ be countable and dense. Then
\[
\mathcal B(C(K))=\sigma\{f\mapsto f(x):x\in S\},
\]
where $C(K)$ is equipped with the uniform norm.
\end{lemma}

\begin{proof}
Each evaluation map $f\mapsto f(x)$ is continuous on $C(K)$, so the right-hand side is contained in $\mathcal B(C(K))$.

For the reverse inclusion, fix $g\in C(K)$ and $\varepsilon>0$. Because $S$ is dense and both $f$ and $g$ are continuous,
\[
\norm{f-g}_\infty = \sup_{x\in K}\abs{f(x)-g(x)} = \sup_{x\in S}\abs{f(x)-g(x)}.
\]
Hence the open ball $B(g,\varepsilon)$ can be written as
\[
B(g,\varepsilon)=\bigcap_{x\in S}\{f\in C(K): \abs{f(x)-g(x)}<\varepsilon\},
\]
which belongs to $\sigma\{f\mapsto f(x):x\in S\}$. Since open balls generate the Borel sigma-field, the claim follows.
\end{proof}

\begin{corollary}\label{cor:dense-observations-from-paths}
Let $S\subset D$ be countable and dense. The law of the original coordinate process $\{Y_j(t):t\in S\}$ equals the pushforward of $\mathbf P_j$ under the evaluation map $E_S:C(\overline D)\to\R^S$, $E_S(f)=(f(t))_{t\in S}$. In particular, if $\mathbf P_1\perp \mathbf P_2$ on $C(\overline D)$, then $\mathbf P_1^S\perp \mathbf P_2^S$ on $\R^S$.
\end{corollary}

\begin{proof}
Because $S$ is countable, the continuous modification and the original process agree simultaneously at every point of $S$ on an event of probability one. Therefore the induced laws on $\R^S$ are identical.

Since $\overline{S} \supseteq \overline{D}$, the set $S$ is dense in $\overline{D}$. If $\mathbf P_1\perp \mathbf P_2$, let $A\in\mathcal B(C(\overline D))$ satisfy $\mathbf P_1(A)=1$ and $\mathbf P_2(A)=0$. By Lemma~\ref{lem:borel-generated}, there exists $B\in\mathcal B(\R^S)$ such that $A=E_S^{-1}(B)$. Hence
\[
\mathbf P_1^S(B)=\mathbf P_1(A)=1,
\qquad
\mathbf P_2^S(B)=\mathbf P_2(A)=0,
\]
so $\mathbf P_1^S\perp \mathbf P_2^S$.
\end{proof}

Thus it remains to prove $\mathbf P_1\perp \mathbf P_2$ on $C(\overline D)$.

\section{Localized Fourier coefficients}

Choose a nonzero real-valued cutoff $\chi\in C_c^\infty(\operatorname{int}D)$ where $\operatorname{int}D$
denotes the interior of $D$.
Thus $\chi$ is infinitely differentiable, compactly supported, and vanishes near the boundary $\partial D$.
Its Fourier transform
\[
\widehat{\chi}(\eta)=\int_{\mathbb R^4} e^{-i\eta\cdot t}\chi(t)\,dt
\]
is a Schwartz function \cite[Theorem 1.13]{duo2024fourier}: for every $M\ge 0$ there exists $C_M<\infty$ such that
\begin{equation}\label{eqn:rapid}
|\widehat{\chi}(\eta)|\le C_M(1+|\eta|)^{-M},
\qquad \eta\in\mathbb R^4.
\end{equation}
This rapid decay is the key localization input in the proof, since it forces the covariances between
different frequency coefficients to decay rapidly away from the diagonal.

For $k\in\N^4$ and a Gaussian random field $Y$ recall the localized Fourier coefficient from from \eqref{eq:Xk-def}:
\begin{equation}
X_k := X_k(Y) = \int_D \chi(t)e^{-ik\cdot t}Y(t)\,dt .
\end{equation}
For each $k$, the coefficient $X_k$ is a mean-zero complex Gaussian random variable.
Write
\begin{equation}\label{eqn:sigmapi}
\Sigma_j(k,\ell):=\E_j\big[X_k\overline{X_\ell}\big],
\qquad
\Pi_j(k,\ell):=\E_j\big[X_k X_\ell\big].
\end{equation}
Set
\[
v_j(k):=\Sigma_j(k,k)=\E_j\abs{X_k}^2.
\]
\begin{lemma}\label{lem:cov-formulas}
For every $k,\ell\in\N^4$,
\begin{align}
\Sigma_j(k,\ell)
&=(2\pi)^{-4}\int_{\R^4} f_j(\xi)\,\what\chi(k-\xi)\,\overline{\what\chi(\ell-\xi)}\,d\xi,
\label{eq:Sigma-formula}
\\
\Pi_j(k,\ell)
&=(2\pi)^{-4}\int_{\R^4} f_j(\xi)\,\what\chi(k-\xi)\,\what\chi(\ell+\xi)\,d\xi,
\label{eq:Pi-formula}
\\
v_j(k)
&=(2\pi)^{-4}\int_{\R^4} f_j(k-\eta)\,\abs{\what\chi(\eta)}^2\,d\eta.
\label{eq:vi-formula}
\end{align}
\end{lemma}

\begin{proof}
Let $Y_j$ denote the original stationary field on $\R^4$, and let $\widetilde Y_j$
be a continuous modification whose law on $D$ is $\mathbf P_j$.
Define the covariance function of the stationary field by
\[
R_j(h):=\E\!\big[Y_j(t+h)Y_j(t)\big], \qquad h\in\R^4.
\]
By stationarity this does not depend on $t$, and since $\widetilde Y_j$ is a
modification of $Y_j$, we have
\[
\E\!\big[\widetilde Y_j(t)\widetilde Y_j(s)\big]=R_j(t-s),
\qquad t,s\in D.
\]
Because $f_j\in L^1(\R^4)$, Fourier inversion gives
\[
R_j(h)=(2\pi)^{-4}\int_{\R^4} e^{i\xi\cdot h}f_j(\xi)\,d\xi.
\]
For each sample path of $\widetilde Y_j$, the integral
\[
X_k(\widetilde Y_j)=\int_D \chi(t)e^{-ik\cdot t}\widetilde Y_j(t)\,dt
\]
is well defined, since $\chi$ is smooth with compact support in $D$ and
$\widetilde Y_j$ is continuous on $\overline D$.

Now
\begin{align*}
\Sigma_j(k,\ell)
&=\E_j\!\left[\left(\int_D \chi(t)e^{-ik\cdot t}\widetilde Y_j(t)\,dt\right)
\overline{\left(\int_D \chi(s)e^{-i\ell\cdot s}\widetilde Y_j(s)\,ds\right)}\right]
\\
&=\iint_{D\times D}
\chi(t)\chi(s)e^{-ik\cdot t}e^{i\ell\cdot s}
\,\E\!\big[\widetilde Y_j(t)\widetilde Y_j(s)\big]\,dt\,ds
\\
&=\iint_{D\times D}
\chi(t)\chi(s)e^{-ik\cdot t}e^{i\ell\cdot s}
R_j(t-s)\,dt\,ds
\\
&=(2\pi)^{-4}\int_{\R^4} f_j(\xi)
\left(\int_D \chi(t)e^{-i(k-\xi)\cdot t}\,dt\right)
\left(\int_D \chi(s)e^{i(\ell-\xi)\cdot s}\,ds\right)\,d\xi
\end{align*}
where in the penultimate step we interchange expectation and integration by Tonelli--Fubini, since the integrand is absolutely integrable over $(D\times D)$.
Since $\operatorname{supp}\chi \subset D$, 
\[
\int_D \chi(t)e^{-i(k-\xi)\cdot t}\,dt
=
\int_{\R^4} \chi(t)e^{-i(k-\xi)\cdot t}\,dt
=
\what\chi(k-\xi).
\]
Since $\chi$ is real-valued,
\[
\int_D \chi(s)e^{i(\ell-\xi)\cdot s}\,ds
=\overline{\int_D \chi(s)e^{-i(\ell-\xi)\cdot s}\,ds}
=\overline{\what\chi(\ell-\xi)}.
\]
This proves \eqref{eq:Sigma-formula}.

The proof of \eqref{eq:Pi-formula} is the same, except that there is no complex
conjugation on the second factor:
\begin{align*}
\Pi_j(k,\ell)
&=\E_j\!\left[\left(\int_D \chi(t)e^{-ik\cdot t}\widetilde Y_j(t)\,dt\right)
\left(\int_D \chi(s)e^{-i\ell\cdot s}\widetilde Y_j(s)\,ds\right)\right]
\\
&=\iint_{D\times D}
\chi(t)\chi(s)e^{-ik\cdot t}e^{-i\ell\cdot s}
R_j(t-s)\,dt\,ds
\\
&=(2\pi)^{-4}\int_{\R^4} f_j(\xi)
\left(\int_D \chi(t)e^{-i(k-\xi)\cdot t}\,dt\right)
\left(\int_D \chi(s)e^{-i(\ell+\xi)\cdot s}\,ds\right)\,d\xi
\\
&=(2\pi)^{-4}\int_{\R^4} f_j(\xi)\,\what\chi(k-\xi)\,\what\chi(\ell+\xi)\,d\xi.
\end{align*}

Finally, setting $\ell=k$ in \eqref{eq:Sigma-formula} gives
\[
v_j(k)=(2\pi)^{-4}\int_{\R^4} f_j(\xi)\,\abs{\what\chi(k-\xi)}^2\,d\xi,
\]
and the change of variables $\eta=k-\xi$ yields \eqref{eq:vi-formula}.
\end{proof}
\section{Diagonal asymptotics}
Let
\begin{equation} \label{eqn:h}
h(\xi):=f_2(\xi)-f_1(\xi).
\end{equation}

\begin{lemma}\label{lem:diagonal-asymptotics}
Define
\[
c_\chi:=(2\pi)^{-4}\int_{\R^4}\abs{\what\chi(\eta)}^2\,d\eta = \int_D \chi(t)^2\,dt >0.
\]
Then, as $\abs{k}\to\infty$ with $k\in\N^4$,
\begin{align}
v_j(k)&=c_\chi f_j(k)(1+o(1)),
\label{eq:vi-asymptotic}
\\
v_2(k)-v_1(k)&=c_\chi h(k)(1+o(1)).
\label{eq:diff-asymptotic}
\end{align}
Consequently, if
\begin{equation}\label{eq:delta-def}
\delta_k:=\frac{v_2(k)}{v_1(k)}-1,
\end{equation}
then
\begin{equation}\label{eq:delta-asymptotic}
\delta_k = p(\alpha_1^2-\alpha_2^2)\abs{k}^{-2}+o(\abs{k}^{-2}).
\end{equation}
In particular, $\abs{\delta_k}\asymp \abs{k}^{-2}$ for all sufficiently large $k$.
\end{lemma}

\begin{proof}
Recall from \eqref{eq:vi-formula} that
\[
v_j(k)
=
(2\pi)^{-4}\int_{\R^4} f_j(k-\eta)\,|\widehat{\chi}(\eta)|^2\,d\eta.
\]
Therefore
\[
\frac{v_j(k)}{f_j(k)}
=
(2\pi)^{-4}\int_{\R^4}\frac{f_j(k-\eta)}{f_j(k)}\,|\widehat{\chi}(\eta)|^2\,d\eta.
\]
We will show that the integrand converges pointwise to $|\widehat{\chi}(\eta)|^2$
and is dominated by an integrable function independent of $k$.

For fixed $\eta\in\R^4$, since
\[
f_j(\xi)=C_\nu\,\sigma_j^2\alpha_j^{2\nu}(\alpha_j^2+|\xi|^2)^{-p},
\qquad p=\nu+2,
\]
we have
\[
\frac{f_j(k-\eta)}{f_j(k)}
=
\left(\frac{\alpha_j^2+|k|^2}{\alpha_j^2+|k-\eta|^2}\right)^p.
\]
If $\eta$ is fixed and $|k|\to\infty$, then
\[
\frac{\alpha_j^2+|k-\eta|^2}{\alpha_j^2+|k|^2}\to 1,
\]
and hence
\[
\frac{f_j(k-\eta)}{f_j(k)}\to 1.
\]

We now establish a uniform bound. Since
\[
f_j(\xi)\asymp (1+|\xi|)^{-2p},
\]
there exists a constant $C_0\ge 1$, depending only on $j$, such that
\[
C_0^{-1}(1+|\xi|)^{-2p}\le f_j(\xi)\le C_0(1+|\xi|)^{-2p}
\qquad\text{for all }\xi\in\R^4.
\]
Therefore
\[
\frac{f_j(k-\eta)}{f_j(k)}
\le
C_0^2\Bigl(\frac{1+|k|}{1+|k-\eta|}\Bigr)^{2p}.
\]
Using the triangle inequality,
\[
1+|k|\le 1+|k-\eta|+|\eta|
\le (1+|k-\eta|)(1+|\eta|),
\]
and hence
\[
\frac{1+|k|}{1+|k-\eta|}
\le 1+|\eta|.
\]
It follows that
\[
\frac{f_j(k-\eta)}{f_j(k)}
\le C(1+|\eta|)^{2p}
\]
for all $k,\eta$, with $C$ independent of $k,\eta$.
From \eqref{eqn:rapid}, for every $M>0$ there exists $C_M$ such that
\[
|\widehat{\chi}(\eta)|\le C_M(1+|\eta|)^{-M}.
\]
Choosing $M$ so large that
\[
(1+|\eta|)^{2p}|\widehat{\chi}(\eta)|^2\in L^1(\R^4),
\]
we may apply dominated convergence to obtain
\[
\frac{v_j(k)}{f_j(k)}
\to
(2\pi)^{-4}\int_{\R^4}|\widehat{\chi}(\eta)|^2\,d\eta
=:c_\chi.
\]
Thus
\[
v_j(k)=c_\chi f_j(k)(1+o(1)).
\]

Finally, by the Plancherel theorem \cite[p.15]{duo2024fourier},
\[
c_\chi=(2\pi)^{-4}\int_{\R^4}|\widehat{\chi}(\eta)|^2\,d\eta
=\int_{\R^4}|\chi(t)|^2\,dt.
\]
This proves \eqref{eq:vi-asymptotic}.

\medskip
\noindent
Next we give the proof of \eqref{eq:diff-asymptotic}.
Under the microergodic matching condition \eqref{eq:micro-match},
\[
f_j(\xi)=C_\nu m(\alpha_j^2+|\xi|^2)^{-p},
\]
so
\[
h(\xi):=f_2(\xi)-f_1(\xi)
=
C_\nu m\Big[(\alpha_2^2+|\xi|^2)^{-p}-(\alpha_1^2+|\xi|^2)^{-p}\Big].
\]

We first expand this quantity as $|\xi|\to\infty$. Set $r=|\xi|^2$, so that
\[
(\alpha_j^2+r)^{-p}=r^{-p}\Bigl(1+\frac{\alpha_j^2}{r}\Bigr)^{-p}.
\]
Using Taylor expansion,
\[
(\alpha_j^2+r)^{-p}
=
r^{-p}\left(1-p\frac{\alpha_j^2}{r}+O(r^{-2})\right)
=
r^{-p}-p\alpha_j^2 r^{-p-1}+O(r^{-p-2}).
\]
Subtracting the two expansions gives
\[
(\alpha_2^2+r)^{-p}-(\alpha_1^2+r)^{-p}
=
p(\alpha_1^2-\alpha_2^2)r^{-p-1}+O(r^{-p-2}).
\]
Thus
\begin{equation}\label{eq:h-expansion-detailed}
h(\xi)
=
C_\nu m\,p(\alpha_1^2-\alpha_2^2)|\xi|^{-2p-2}
+
O(|\xi|^{-2p-4}).
\end{equation}

Since $\alpha_1\ne\alpha_2$, there exist $R>0$ and constants $c,C>0$ such that
\[
c(1+|\xi|)^{-2p-2}\le |h(\xi)|\le C(1+|\xi|)^{-2p-2}
\qquad\text{for }|\xi|\ge R.
\]
By enlarging $C$ if necessary, we may assume the upper bound holds for all $\xi$.

Now recall from \eqref{eq:vi-formula} that
\[
v_2(k)-v_1(k)
=
(2\pi)^{-4}\int_{\R^4}h(k-\eta)\,|\widehat{\chi}(\eta)|^2\,d\eta.
\]
Therefore, 
\[
\frac{v_2(k)-v_1(k)}{h(k)}
=
(2\pi)^{-4}\int_{\R^4}\frac{h(k-\eta)}{h(k)}\,|\widehat{\chi}(\eta)|^2\,d\eta.
\]
For fixed \(\eta\) we have 
\[
\frac{h(k-\eta)}{h(k)}\to 1
\qquad\text{as }|k|\to\infty.
\]
Since
\[
|h(\zeta)|\le C(1+|\zeta|)^{-2p-2},
\]
while for large $|k|$,
\[
|h(k)|\ge c(1+|k|)^{-2p-2},
\]
it follows that for all sufficiently large $k$,
\[
\frac{|h(k-\eta)|}{|h(k)|}
\le
C'\Bigl(\frac{1+|k|}{1+|k-\eta|}\Bigr)^{2p+2}.
\]
Using again
\[
1+|k|\le (1+|k-\eta|)(1+|\eta|),
\]
we conclude that
\[
\frac{|h(k-\eta)|}{|h(k)|}
\le
C'(1+|\eta|)^{2p+2}
\]
for all sufficiently large \(k\), uniformly in \(\eta\).
Since \(\widehat{\chi}\) is Schwartz,
\[
(1+|\eta|)^{2p+2}|\widehat{\chi}(\eta)|^2\in L^1(\R^4).
\]
Thus dominated convergence applies and yields
\[
\frac{v_2(k)-v_1(k)}{h(k)}
\to
(2\pi)^{-4}\int_{\R^4}|\widehat{\chi}(\eta)|^2\,d\eta
=
c_\chi.
\]
Equivalently,
\[
v_2(k)-v_1(k)=c_\chi h(k)(1+o(1)).
\]
This proves \eqref{eq:diff-asymptotic}.

By \eqref{eq:diff-asymptotic} and \eqref{eq:vi-asymptotic}, we may write
\[
v_2(k)-v_1(k)=c_\chi h(k)+o(h(k)),
\qquad
v_1(k)=c_\chi f_1(k)+o(f_1(k)).
\]
Therefore
\[
\delta_k
=
\frac{v_2(k)-v_1(k)}{v_1(k)}
=
\frac{c_\chi h(k)+o(h(k))}{c_\chi f_1(k)+o(f_1(k))}.
\]
Since \(f_1(k)>0\) for all \(k\), and
\[
c_\chi f_1(k)+o(f_1(k))\sim c_\chi f_1(k),
\]
it follows that
\[
\frac{1}{c_\chi f_1(k)+o(f_1(k))}
=
\frac{1}{c_\chi f_1(k)}(1+o(1)).
\]
Hence
\[
\delta_k
=
\bigl(c_\chi h(k)+o(h(k))\bigr)\frac{1}{c_\chi f_1(k)}(1+o(1))
=
\frac{h(k)}{f_1(k)}+o\!\left(\frac{h(k)}{f_1(k)}\right).
\]
Equivalently,
\[
\delta_k
=
\frac{h(k)}{f_1(k)}\bigl(1+o(1)\bigr).
\]
Since, by the expansion of \(h\),
\[
\frac{h(k)}{f_1(k)}\asymp |k|^{-2},
\]
we obtain the sharper form
\[
\delta_k
=
\frac{h(k)}{f_1(k)}+o(|k|^{-2}).
\]
Now
\[
\frac{h(k)}{f_1(k)}
=
\frac{f_2(k)}{f_1(k)}-1
=
\left(\frac{\alpha_1^2+|k|^2}{\alpha_2^2+|k|^2}\right)^p-1.
\]
Writing
\[
u_k:=\frac{\alpha_1^2-\alpha_2^2}{\alpha_2^2+|k|^2},
\]
we have \(u_k=O(|k|^{-2})\), and therefore
\[
\frac{f_2(k)}{f_1(k)}
=
(1+u_k)^p
=
1+p\,u_k+O(u_k^2)
=
1+p(\alpha_1^2-\alpha_2^2)|k|^{-2}+O(|k|^{-4}).
\]
Thus
\[
\frac{h(k)}{f_1(k)}
=
p(\alpha_1^2-\alpha_2^2)|k|^{-2}+O(|k|^{-4}),
\]
and consequently
\[
\delta_k
=
p(\alpha_1^2-\alpha_2^2)|k|^{-2}+o(|k|^{-2}).
\]
This shows \eqref{eq:delta-asymptotic} and completes the proof.
\end{proof}
\begin{lemma}\label{lem:log-growth}
There exists $K_0\in\N$ such that for
\[
\Lambda_N:=\{k\in\N^4: K_0\le \norm{k}_\infty\le N\}
\qquad\text{and}\qquad
L_N:=\sum_{k\in\Lambda_N} \delta_k^2,
\]
one has
\begin{equation}\label{eq:LN-log}
L_N\asymp \log N\qquad (N\to\infty).
\end{equation}
In particular $L_N\to\infty$.
\end{lemma}

\begin{proof}
By Lemma~\ref{lem:diagonal-asymptotics}, after increasing $K_0$ if necessary,
\[
\abs{\delta_k}\asymp \norm{k}_\infty^{-2}\qquad\text{for all }k\in\Lambda_N.
\]
Hence
\[
L_N\asymp \sum_{K_0\le \norm{k}_\infty\le N} \norm{k}_\infty^{-4}.
\]
Now
\[
\#\{k\in\N^4:\norm{k}_\infty=n\}=n^4-(n-1)^4\asymp n^3.
\]
Therefore
\[
L_N\asymp \sum_{n=K_0}^N n^3\,n^{-4} \asymp \sum_{n=K_0}^N \frac1n \asymp \log N
\]
 and we are done.
\end{proof}

\section{Off-diagonal decay}
Define the rapidly decreasing nonnegative function
\[
a(\eta):=(1+\abs{\eta})^p\abs{\what\chi(\eta)},\qquad \eta\in\R^4.
\]
For $q,r\in\R^4$, define
\begin{align}
A(q)&:=\int_{\R^4} a(\eta)a(\eta+q)\,d\eta,
\label{eq:A-def}
\\
B(r)&:=\int_{\R^4} a(\eta)a(r-\eta)\,d\eta.
\label{eq:B-def}
\end{align}

\begin{lemma}\label{lem:A-B-rapid}
For every $M>0$ there exists $C_M<\infty$ such that
\begin{equation}\label{eq:A-B-rapid}
A(q)+B(q)\le C_M(1+\abs{q})^{-M}\qquad\text{for all }q\in\R^4.
\end{equation}
In particular, the restrictions of $A^2$ and $B^2$ to $\Z^4$ belong to $\ell^1(\Z^4)$.
\end{lemma}

\begin{proof}
It suffices to show that the convolution of two rapidly decreasing functions is rapidly decreasing. Fix $M>0$ and choose $L>M+4$. Since $a$ is rapidly decreasing,
\[
a(\eta)\le C_L(1+|\eta|)^{-L}.
\]
Hence
\[
A(q)\le C_L^2 \int_{\R^4}(1+|\eta|)^{-L}(1+|\eta+q|)^{-L}\,d\eta
=: C_L^2\, I_L(q),
\]
where $L>4$ ensures integrability. Set
\[
E_1:=\{\eta:|\eta|\ge |q|/2\},
\qquad
E_2:=\{\eta:|\eta|<|q|/2\}.
\]
Then $\R^4=E_1\sqcup E_2$, and on $E_2$ we have $|\eta+q|\ge |q|-|\eta|>|q|/2$. Thus
\[
I_L(q)=\int_{E_1}\cdots\,d\eta+\int_{E_2}\cdots\,d\eta=:I_{L,1}(q)+I_{L,2}(q).
\]
On $E_1$,
\[
(1+|\eta|)^{-L}\le (1+|q|/2)^{-L},
\]
so
\[
I_{L,1}(q)
\le
(1+|q|/2)^{-L}\int_{\R^4}(1+|\eta+q|)^{-L}\,d\eta
\le C_L(1+|q|)^{-L}.
\]
On $E_2$, since $|\eta+q|>|q|/2$,
\[
(1+|\eta+q|)^{-L}\le (1+|q|/2)^{-L},
\]
and hence
\[
I_{L,2}(q)
\le
(1+|q|/2)^{-L}\int_{\R^4}(1+|\eta|)^{-L}\,d\eta
\le C_L(1+|q|)^{-L}.
\]
Therefore
\[
I_L(q)\le C_L(1+|q|)^{-L},
\]
and consequently
\[
A(q)\le C(1+|q|)^{-L}\le C(1+|q|)^{-M}.
\]
The proof for $B$ is identical. Since $M$ is arbitrary, \eqref{eq:A-B-rapid} holds. 
For \(q\in \mathbb Z^4\), 
\[
A(q)^2 \le C_M^2 (1+|q|)^{-2M},
\qquad
B(q)^2 \le C_M^2 (1+|q|)^{-2M}.
\]
Hence to show that the restrictions of $A^2$ and $B^2$ to $\Z^4$ belong to $\ell^1(\Z^4)$, it suffices to show that
\[
\sum_{q\in \mathbb Z^4} (1+|q|)^{-2M}<\infty.
\]
For \(S_n=\{q\in\mathbb Z^4:n\le |q|<n+1\}\),
\[
\#S_n
=
\#(B_{n+1}\cap\mathbb Z^4)-\#(B_n\cap\mathbb Z^4),
\]
where \(B_r=\{x\in\mathbb R^4:|x|\le r\}\).
Since \(B_r\subset[-r,r]^4\),
\[
\#(B_r\cap\mathbb Z^4)\le (2r+1)^4\lesssim r^4.
\]
Hence
\[
\#S_n\lesssim (n+1)^4-n^4\lesssim n^3 
\] 
and thus \[
\sum_{q\in \mathbb Z^4}(1+|q|)^{-2M}
\le
\sum_{n=0}^\infty \#S_n (1+n)^{-2M}
\lesssim
\sum_{n=0}^\infty (1+n)^{3-2M}.
\]
This converges whenever \(2M>4\); thus taking \(M>4\) yields
\[
A^2|_{\mathbb Z^4},\,B^2|_{\mathbb Z^4}\in \ell^1(\mathbb Z^4).
\]
This completes the proof.
\end{proof}
\begin{lemma}\label{lem:off-diagonal-bounds}
After increasing $K_0$ from \eqref{eqn:Lambda} if necessary, there exists a constant $C<\infty$ such that for all $k,\ell\in\N^4$ with $\min\{\norm{k}_\infty,\norm{\ell}_\infty\}\ge K_0$,
\begin{align}
\abs{\Sigma_j(k,\ell)} &\le C\,A(\ell-k)\,\sqrt{v_j(k)v_j(\ell)},
\label{eq:Sigma-bound}
\\
\abs{\Pi_j(k,\ell)} &\le C\,B(k+\ell)\,\sqrt{v_j(k)v_j(\ell)}.
\label{eq:Pi-bound}
\end{align}
\end{lemma}

\begin{proof}
We prove \eqref{eq:Sigma-bound}; the proof of \eqref{eq:Pi-bound} is analogous.
By \eqref{eq:Sigma-formula}, after the change of variable $\eta=k-\xi$,
\[
\Sigma_j(k,\ell)
=(2\pi)^{-4}\int_{\R^4} f_j(k-\eta)\,\widehat{\chi}(\eta)\,
\overline{\widehat{\chi}(\eta+\ell-k)}\,d\eta .
\]

Since $f_j(\xi)\asymp (1+|\xi|)^{-2p}$, we have
\[
f_j(k-\eta)\le C(1+|\eta|)^{2p} f_j(k)
\]
and also, since $k-\eta=\ell-(\eta+\ell-k)$,
\[
f_j(k-\eta)\le C(1+|\eta+\ell-k|)^{2p} f_j(\ell).
\]
Therefore
\begin{align*}
f_j(k-\eta)
&\le C\min\Big\{(1+|\eta|)^{2p}f_j(k),\,
(1+|\eta+\ell-k|)^{2p}f_j(\ell)\Big\}
\\
&\le C(1+|\eta|)^p(1+|\eta+\ell-k|)^p\sqrt{f_j(k)f_j(\ell)} .
\end{align*}
Hence
\[
|\Sigma_j(k,\ell)|
\le
C\sqrt{f_j(k)f_j(\ell)}
\int_{\R^4} a(\eta)a(\eta+\ell-k)\,d\eta
=
C A(\ell-k)\sqrt{f_j(k)f_j(\ell)} .
\]
By Lemma~\ref{lem:diagonal-asymptotics}, $v_j(k)\asymp f_j(k)$ for $k\in\N^4$
with $\norm{k}_\infty\ge K_0$, so \eqref{eq:Sigma-bound} follows.

For \eqref{eq:Pi-bound}, start from
\[
\Pi_j(k,\ell)
=(2\pi)^{-4}\int_{\R^4} f_j(k-\eta)\,\widehat{\chi}(\eta)\,
\widehat{\chi}(k+\ell-\eta)\,d\eta .
\]
The same estimate as above gives
\[
f_j(k-\eta)
\le
C(1+|\eta|)^p(1+|k+\ell-\eta|)^p\sqrt{f_j(k)f_j(\ell)},
\]
because $f_j$ is radial and $k-\eta=-\ell+(k+\ell-\eta)$.
Therefore
\[
|\Pi_j(k,\ell)|
\le
C B(k+\ell)\sqrt{f_j(k)f_j(\ell)}
\le
C B(k+\ell)\sqrt{v_j(k)v_j(\ell)},
\]
which completes the proof.
\end{proof}

\section{Separation of models via $T_N$}
Recall the constant $K_0$ from \eqref{eqn:Lambda} and 
for $N\ge K_0$ define
\begin{equation}\label{eq:TN-def}
T_N:=\frac1{L_N}\sum_{k\in\Lambda_N} \delta_k\left(\frac{\abs{X_k}^2}{v_1(k)}-1\right).
\end{equation}

\begin{lemma}\label{lem:means}
For every $N\ge K_0$,
\[
\E_1[T_N]=0,
\qquad
\E_2[T_N]=1.
\]
\end{lemma}

\begin{proof}
Under $\mathbf P_1$, $\E_1\abs{X_k}^2=v_1(k)$, so $\E_1[T_N]=0$.
Under $\mathbf P_2$,
\[
\E_2\left(\frac{\abs{X_k}^2}{v_1(k)}-1\right)=\frac{v_2(k)}{v_1(k)}-1=\delta_k.
\]
Hence
\[
\E_2[T_N]=\frac1{L_N}\sum_{k\in\Lambda_N}\delta_k^2=1
\]
and we are done.
\end{proof}

\begin{lemma}\label{lem:complex-wick}
Let $Z,W$ be centered complex Gaussian random variables. Then
\[
\operatorname{Cov}(\abs{Z}^2,\abs{W}^2)=\abs{\E[Z\overline W]}^2+\abs{\E[ZW]}^2.
\]
\end{lemma}

\begin{proof}
This follows from Wick's (Isserlis') formula for centered Gaussian variables (see, \textit{e.g.}, \cite[Theorem~1.28]{Janson1997}), which asserts that for centered Gaussian random variables $X_1,\dots,X_4$,
\[
\E[X_1X_2X_3X_4]
= \E[X_1X_2]\E[X_3X_4]
+ \E[X_1X_3]\E[X_2X_4]
+ \E[X_1X_4]\E[X_2X_3].
\]
Applying this with $(X_1,X_2,X_3,X_4)=(Z,\overline Z,W,\overline W)$ yields
\[
\E[\abs{Z}^2\abs{W}^2]
= \E[Z\overline Z]\E[W\overline W]
+ \E[Z\overline W]\E[\overline Z W]
+ \E[ZW]\E[\overline Z\,\overline W].
\]
Subtracting $\E\abs{Z}^2\,\E\abs{W}^2$ gives the claim. 
\end{proof}
\begin{proposition}\label{prop:variance}
There exists $C<\infty$ such that for $j=1,2$ and all $N\ge K_0$,
\begin{equation}\label{eq:variance-bound}
\operatorname{Var}_j(T_N)\le \frac{C}{L_N}.
\end{equation}
Consequently, along any sequence $N_s\to\infty$ with $\sum_s L_{N_s}^{-1}<\infty$, one has
\[
T_{N_s}-\E_j[T_{N_s}]\longrightarrow 0
\qquad \mathbf P_j\text{-almost surely.}
\]
\end{proposition}

\begin{proof}
Since $T_N$ is centered under $\mathbf P_1$ and shifted by its mean under $\mathbf P_2$, the variance under either measure is
\[
\operatorname{Var}_j(T_N)
=\frac1{L_N^2}\sum_{k,\ell\in\Lambda_N}
\frac{\delta_k\delta_\ell}{v_1(k)v_1(\ell)}
\operatorname{Cov}_j\big(\abs{X_k}^2,\abs{X_\ell}^2\big).
\]
Taking absolute values and using Lemma~\ref{lem:complex-wick},
\[
\operatorname{Var}_j(T_N)
\le \frac1{L_N^2}\sum_{k,\ell\in\Lambda_N}
\frac{\abs{\delta_k\delta_\ell}}{v_1(k)v_1(\ell)}
\Big(\abs{\Sigma_j(k,\ell)}^2+\abs{\Pi_j(k,\ell)}^2\Big).
\]
By Lemma~\ref{lem:diagonal-asymptotics}, after increasing $K_0$ if necessary,
\[
\frac12\le \frac{v_2(k)}{v_1(k)}\le 2\qquad (k\in\Lambda_N),
\]
so $v_j(k)/v_1(k)$ is uniformly bounded above and below for $j=1,2$. 
Combining this with Lemma~\ref{lem:off-diagonal-bounds} yields
\begin{align}
\operatorname{Var}_j(T_N)
&\le \frac{C}{L_N^2}\sum_{k,\ell\in\Lambda_N}\abs{\delta_k\delta_\ell}
\Big(A(\ell-k)^2+B(k+\ell)^2\Big) \notag\\
&=: \frac{C}{L_N^2}\big(I_{N,1}+I_{N,2}\big),
\label{eq:var-split}
\end{align}
where
\begin{align*}
I_{N,1}
&:= \sum_{k,\ell\in\Lambda_N}\abs{\delta_k\delta_\ell}\,A(\ell-k)^2,\\
I_{N,2}
&:= \sum_{k,\ell\in\Lambda_N}\abs{\delta_k\delta_\ell}\,B(k+\ell)^2.
\end{align*}

For estimating $I_{N,1}$, define $a_q:=A(q)^2$ on $\Z^4$ and extend $\delta_k$ by zero outside $\Lambda_N$. Since $a\in\ell^1(\Z^4)$ by Lemma~\ref{lem:A-B-rapid},
\begin{align*}
\sum_{k,\ell\in\Lambda_N}\abs{\delta_k\delta_\ell}A(\ell-k)^2
&=\sum_{k\in\Z^4}\abs{\delta_k}(a*\abs{\delta})(k)
\\
&\le \|\delta\|_{\ell^2}\,\|a*\abs{\delta}\|_{\ell^2}
\\
&\le \|a\|_{\ell^1}\,\|\delta\|_{\ell^2}^2
=\|a\|_{\ell^1}L_N,
\end{align*}
where the second line is Cauchy--Schwarz, and the third line is Young's convolution inequality
\[
\|a*\abs{\delta}\|_{\ell^2(\mathbb Z^4)}
\le
\|a\|_{\ell^1(\mathbb Z^4)}\,\|\delta\|_{\ell^2(\mathbb Z^4)}.
\]

The estimate for \(I_{N,1}\) uses the convolution structure coming from the
difference variable \(\ell-k\). No analogous convolution argument is available
for \(I_{N,2}\), since \(B(k+\ell)^2\) depends on the sum rather than the
difference; instead we exploit the positivity of \(k,\ell\in\mathbb N^4\) to
separate the decay in \(k+\ell\) into a product of one-variable weights.
This is also one of the reasons for indexing the statistic by $\N^4$ rather than by all of $\Z^4$.

Fix $M>4$. By Lemma~\ref{lem:A-B-rapid},
\[
B(k+\ell)^2\le C_M(1+\norm{k+\ell}_\infty)^{-2M}.
\]
Since $k,\ell\in\N^4$,
\[
\norm{k+\ell}_\infty \ge \frac14\big(\norm{k}_\infty+\norm{\ell}_\infty\big),
\]
because $\norm{x}_1\le 4\norm{x}_\infty$ and $\norm{k+\ell}_1=\norm{k}_1+\norm{\ell}_1\ge \norm{k}_\infty+\norm{\ell}_\infty$. Therefore
\begin{align*}
(1+\norm{k+\ell}_\infty)^{-2M}
&\le C(1+\norm{k}_\infty+\norm{\ell}_\infty)^{-2M} = C\Big((1+\norm{k}_\infty+\norm{\ell}_\infty)^2\Big)^{-M} \\
&\leq C(1+\norm{k}_\infty)^{-M}(1+\norm{\ell}_\infty)^{-M}.
\end{align*}
Hence
\begin{equation} \label{eqn:in2}
\sum_{k,\ell\in\Lambda_N}\abs{\delta_k\delta_\ell}B(k+\ell)^2
\le C\left(\sum_{k\in\N^4}\abs{\delta_k}(1+\norm{k}_\infty)^{-M}\right)^2.
\end{equation}
By Lemma~\ref{lem:diagonal-asymptotics}, after enlarging \(K_0\) if necessary,
$
|\delta_k| \asymp |k|^{-2}, k\in\Lambda_N.
$
Since \(\|k\|_\infty \le |k| \le 2\|k\|_\infty\) in \(\mathbb R^4\), it follows that
\[
|\delta_k|\lesssim \|k\|_\infty^{-2}, \qquad k\in\Lambda_N.
\]
Thus the series on the right-hand side of \eqref{eqn:in2} converges provided $M>2$; in particular it is finite for our choice $M>4$. Thus we have shown that $I_{N,2}$ is $O(1)$.

Inserting these two bounds into \eqref{eq:var-split} gives
\[
\operatorname{Var}_j(T_N)\le \frac{C(L_N+1)}{L_N^2}\le \frac{C'}{L_N},
\]
which is \eqref{eq:variance-bound}.

Now let $N_s\to\infty$ with $\sum_s L_{N_s}^{-1}<\infty$. By Chebyshev's inequality and \eqref{eq:variance-bound}, for every $\varepsilon>0$,
\[
\sum_{s=1}^\infty \mathbf P_j\big(\abs{T_{N_s}-\E_j[T_{N_s}]}>\varepsilon\big)
\le \frac1{\varepsilon^2}\sum_{s=1}^\infty \operatorname{Var}_j(T_{N_s})<\infty.
\]
Borel-Cantelli implies $T_{N_s}-\E_j[T_{N_s}]\to 0$ almost surely under $\mathbf P_j$.
\end{proof}

\section{Proof of Theorem~\ref{thm:main}}

\begin{theorem}\label{thm:path-singularity}
Under \eqref{eq:micro-match}, the path laws $\mathbf P_1$ and $\mathbf P_2$ on $C(\overline D)$ are mutually singular.
\end{theorem}

\begin{proof}
Let $N_s:=\lceil e^{s^2}\rceil$. By Lemma~\ref{lem:log-growth}, $L_{N_s}\asymp s^2$, so
\[
\sum_{s=1}^\infty \frac1{L_{N_s}}<\infty.
\]
Proposition~\ref{prop:variance} and Lemma~\ref{lem:means} give
\[
T_{N_s}\to 0\quad \mathbf P_1\text{-almost surely},
\qquad
T_{N_s}\to 1\quad \mathbf P_2\text{-almost surely}.
\]
Therefore the event
\[
A:=\left\{f\in C(\overline D): \lim_{s \to\infty} T_{N_s}(f)=0\right\}
\]
satisfies $\mathbf P_1(A)=1$ and $\mathbf P_2(A)=0$. Hence $\mathbf P_1\perp \mathbf P_2$.
\end{proof}

\begin{proof}[Proof of Theorem~\ref{thm:main}]
Theorem~\ref{thm:path-singularity} gives $\mathbf P_1\perp \mathbf P_2$ on $C(\overline D)$. Corollary~\ref{cor:dense-observations-from-paths} then implies $\mathbf P_1^S\perp \mathbf P_2^S$ for every countable dense set $S\subset D$, in particular for $S=D_\infty$ and we are done.
\end{proof}

\section{Estimating parameters via Whittle Likelihood}
Although the main purpose of this paper is to establish mutual singularity in the critical dimension \(d=4\), the same localized Fourier coefficients also suggest a natural route to parameter estimation. 
To illustrate this, we present a simulation study in which the Mat\'ern parameters are estimated via a Whittle-type diagonal Gaussian pseudo-likelihood constructed from discrete Fourier coefficients of the observed field. 

Throughout, the smoothness \(\nu\) is assumed known. For each candidate \(\alpha\), we normalize to the microergodic scale \(m=1\), so that
\[
\sigma^2=\alpha^{-2\nu}.
\]
We simulate a Mat\'ern field \(Y\) on the fixed domain
$
D=[0,2\pi)^4$
and observe it on the regular lattice
\[
\mathcal G_{\mathrm{obs}}
=
\{x_j=h\,j:\ j\in\{0,\dots,n_{\mathrm{obs}}-1\}^4\},
\qquad
h=\frac{2\pi}{n_{\mathrm{obs}}},
\]
with total sample size \(N_{\mathrm{obs}}=n_{\mathrm{obs}}^4\). In this section we switch to exact-grid discrete Fourier coefficients on the observation lattice. These differ from the localized Fourier coefficients used in the proof.

Define the sampled-grid covariance
\[
c_\alpha(r)
=
\operatorname{Cov}\!\bigl(Y(0),Y(hr)\bigr)
=
C_{\alpha,\alpha^{-2\nu},\nu}(h\|r\|),
\qquad r\in\mathbb Z^4.
\]
Let
\[
\omega(j)=\mathbf{1}_{\{0,\dots,n_{\mathrm{obs}}-1\}^4}(j),
\qquad j\in\mathbb Z^4,
\]
and define
\[
A(r)
=
\sum_{j\in\mathbb Z^4}\omega(j)\omega(j+r)
=
\prod_{\ell=1}^4 \bigl(n_{\mathrm{obs}}-|r_\ell|\bigr)_+,
\qquad r\in\mathbb Z^4.
\]
The observed-grid Fourier coefficients are defined by
\[
Z_k
=
h^4\sum_{j\in\{0,\dots,n_{\mathrm{obs}}-1\}^4}
Y(hj)e^{-2\pi i \langle k,j\rangle/n_{\mathrm{obs}}},
\qquad
I_k=|Z_k|^2,
\]
and can be computed efficiently via FFT.
Because the field is real-valued, the Fourier coefficients satisfy
$
Z_{-k} = \overline{Z_k}$.
To avoid double counting, we retain one representative from each pair $\{k,-k\}$ and exclude the zero mode $k=0$. We denote the resulting set of frequencies by $\mathcal K \subset \{0,\dots,n_{\mathrm{obs}}-1\}^4 \setminus \{0\}$. The exact variance of \(Z_k\) under the normalization \(m=1\) is
\[
u_\alpha(k)
=
h^8
\sum_{r\in\mathbb Z^4}
c_\alpha(r)\,A(r)\,
e^{-2\pi i \langle k,r\rangle/n_{\mathrm{obs}}}.
\]
Under a general parameter pair \((m,\alpha)\), this becomes
$
\operatorname{Var}(Z_k)=m\,u_\alpha(k).
$
We then define the Whittle pseudo-likelihood
\[
Q_N(m,\alpha)
=
\sum_{k\in\mathcal K}
\left\{
\log\bigl(m\,u_\alpha(k)\bigr)
+
\frac{I_k}{m\,u_\alpha(k)}
\right\}.
\]
For fixed \(\alpha\), the minimizer in \(m\) is explicit:
\[
\widehat m(\alpha)
=
\frac{1}{|\mathcal K|}
\sum_{k\in\mathcal K}\frac{I_k}{u_\alpha(k)}.
\]
We therefore profile out \(m\) and estimate \(\alpha\) by
\[
\widehat\alpha
=
\arg\min_{\alpha\in[\alpha_{\min},\alpha_{\max}]}
Q_N\bigl(\widehat m(\alpha),\alpha\bigr),
\qquad
\widehat m=\widehat m(\widehat\alpha).
\]

In our implementation, the minimization over \(\alpha\) is carried out by a grid search over a fixed one-dimensional set of candidate values. In our experiment, we take
$
\nu=\frac32,
m=1, \alpha = 3.
$
Figure \ref{fig:whittle_alpha} shows that the Whittle likelihood estimator correctly recovers the parameter $\alpha$ even for moderate sample sizes. The concentration of \(\widehat{\alpha}\) around \(\alpha = 3\) reflects the fact that the Fourier coefficients retain sufficient information for parameter recovery, despite the presence of strong dependence across modes. It is an interesting open problem to show the consistency of the Whittle likelihood estimator for both $m$ and $\alpha$.
\begin{figure}[t]
\centering
\includegraphics[width=1\textwidth]{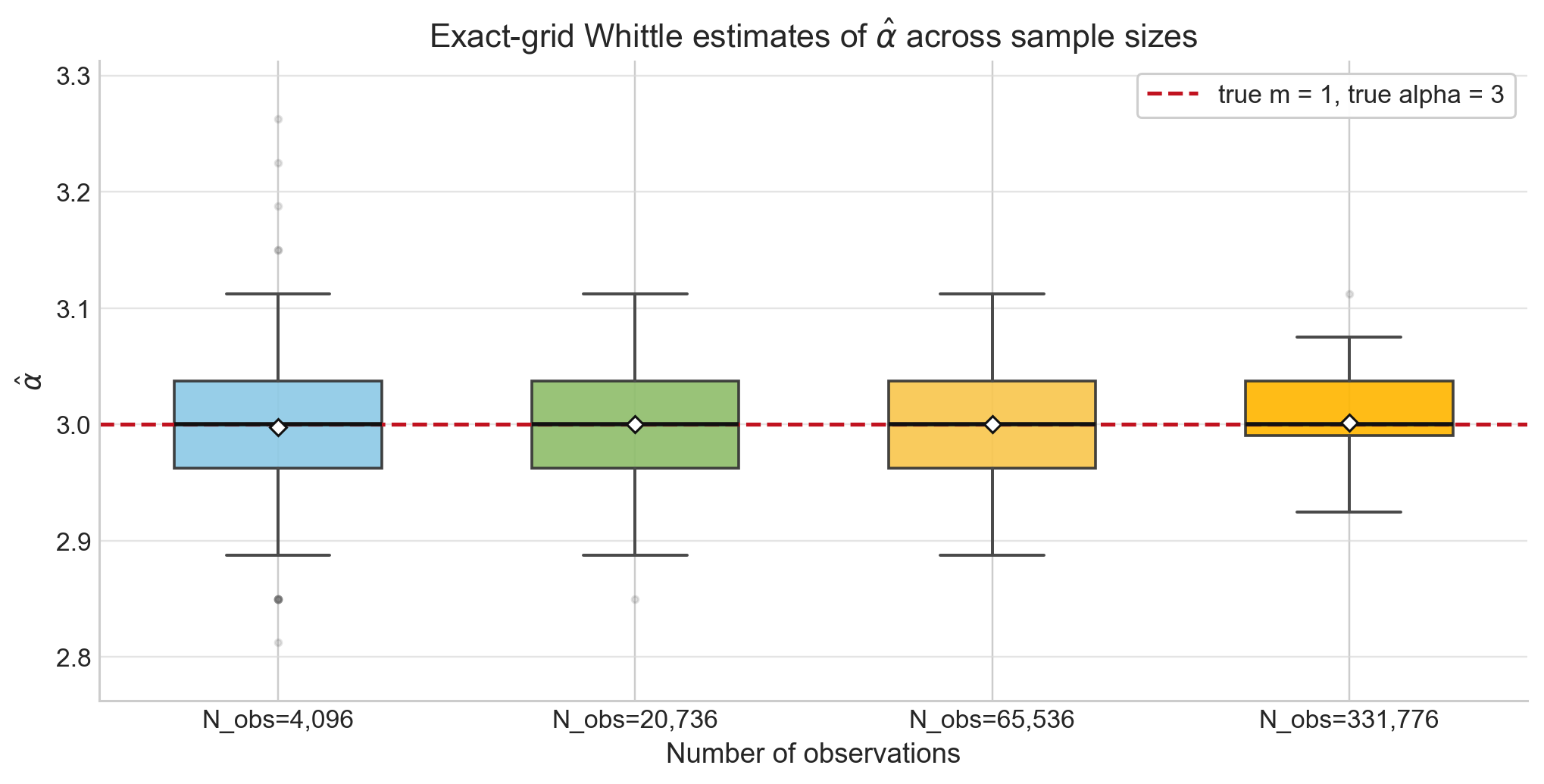}
\caption{
Monte Carlo distribution of the Whittle estimator \(\widehat{\alpha}\) based on the (exact) grid Fourier coefficients in dimension \(d=4\). 
The true parameter value \(\alpha = 3\) is indicated by the dashed line. 
}
\label{fig:whittle_alpha}
\end{figure}

\section*{Acknowledgements}
I thank Sudipto Banerjee, Debdeep Pati, and Aaron Smith for useful discussions, and David Bolin for helpful comments on how the paper \cite{BoliKirc2023} relates to our work. I acknowledge the use of AI (GPT 5.4) at every stage during the preparation of this paper.
\bibliographystyle{unsrtnat}
\bibliography{refs}

%
%

\end{document}